\documentclass[12pt]{amsart}
\usepackage{amsmath,amstext,amsbsy,amssymb,amsfonts,amscd,calc}
\usepackage{color}
\usepackage{hyperref}
\usepackage[utf8]{inputenc}

\usepackage{fullpage}

\usepackage{tikz}
\usetikzlibrary{arrows}

\newtheorem{conjecture}{Conjecture}[section]
\newtheorem{theorem}[conjecture]{Theorem}
\newtheorem{lemma}[conjecture]{Lemma}
\newtheorem{corollary}[conjecture]{Corollary}
\newtheorem{proposition}[conjecture]{Proposition}
\theoremstyle{definition}
\newtheorem{definition}[conjecture]{Definition}
\newtheorem{example}[conjecture]{Example}
\newtheorem{remark}[conjecture]{Remark}
\newtheorem{algorithm}[conjecture]{Algorithm}

\newcommand{\Z}{{\mathbb{Z}}}
\newcommand{\N}{{\Z^{>0}}}
\newcommand{\Rplus}{{\mathbb{R}^{>0}}}
\newcommand{\Q}{{\mathbb{Q}}}
\newcommand{\F}{{\mathbb{F}}}
\newcommand{\Qp}{{\Q_p}}
\newcommand{\calO}{{\mathcal{O}}}
\newcommand{\KK}{{K}}

\newcommand{\clK}{{\overline{K}}} 
\newcommand{\OK}{{\calO_\KK}} 
\newcommand{\RK}{{\underline{\KK}}} 
\newcommand{\repsK}{R_\RK}  
\newcommand{\reps}[1]{R_{#1}}  
\newcommand{\repss}[1]{\reps{#1}^\times}  
\newcommand{\RA}{\underline{A}} 
\newcommand{\RS}[1]{\underline{S}_{#1}} 
\newcommand{\hhf}{\phi} 

\newcommand{\KL}{{L}}
\newcommand{\OL}{{\calO_\KL}}

\newcommand{\KN}{{N}}                 

\newcommand{\KM}{{M}}                 

\newcommand{\disc}{{\operatorname{disc}\,}}

\newcommand{\cont}[2]{\operatorname{cont}_{#1}\!\left(#2\right)}

\newcommand{\norm}{{\mathrm{N}}}

\newcommand{\csmod}{{\;\mathtt{mod}\;}}

\newcommand{\slopeden}{e} 
\newcommand{\slopenum}{h} 

\newcommand{\invA}{\mathcal{A}} 
\newcommand{\npol}{\mathcal{N}} 
\newcommand{\rpol}{\mathcal{R}} 
\newcommand{\spol}{\mathcal{S}} 
\newcommand{\cpol}[1]{\mathcal{N}_{#1}} 
\newcommand{\ppol}{\mathcal{P}} 

\newcommand{\rpolypoints}{\{(1,J_0),(p^{s_1},J_1),\ldots,(p^{s_{u-1}},J_{u-1}),(p^{s_u},0),\ldots,(n,0)\}}
\newcommand{\npl}{i}  

\newcommand{\ote}[1]{{$#1$}} 
\newcommand{\nge}[1]{{{\color{green}$\mathbf{#1}$}}} 
\newcommand{\ngt}[1]{{\color{green}#1}} 
\newcommand{\nre}[1]{{{\color{red}$\mathbf{#1}$}}} 
\newcommand{\nrt}[1]{{\color{red}#1}} 
\newcommand{\nbe}[1]{{{\color{blue}$\mathbf{#1}$}}} 
\newcommand{\nbt}[1]{{\color{blue}#1}} 

\usepackage{enumitem}

\setlist[itemize,1]{label=$\bullet$}
\setlist[itemize,2]{label=$\circ$}
\setlist[itemize,3]{label=$\circ$}
\setlist[itemize,4]{label=$\bullet$}
\setlist[itemize,5]{label=$\bullet$}
\setlist[itemize,6]{label=$\bullet$}
\setlist[itemize,7]{label=$\bullet$}
\setlist[itemize,8]{label=$\bullet$}
\setlist[itemize,9]{label=$\bullet$}

\renewlist{itemize}{itemize}{9}

\setlist[enumerate,1]{label={\rm(\alph*)}}
\setlist[enumerate,2]{label={\rm(\roman*)}}

\newcommand{\Algo}[5]
                {
                \begin{algorithm}[\texttt{#1}] \label{#2}{$\;$}\rm
                    \\[1ex]
\mbox{\enspace}
                    \rlap{\rm Input: }\phantom{\rm Output: }
\parbox[t]{\textwidth-\widthof{{\rm Output: }}-\widthof{\enspace\enspace}}{#3}
                    \\[1ex]
\mbox{\enspace}
                    {\rm Output: }
\parbox[t]{\textwidth-\widthof{{\rm Output: }}-\widthof{\enspace\enspace}}{#4}
\\[1ex]
\parskip0pt
\begin{list}{}{\setlength{\leftmargin}{0pt}}
\item                    #5
\end{list}
\parskip6pt
                \end{algorithm}
                \goodbreak}

\begin{document}

\title{Enumerating Extensions of $(\pi)$-Adic Fields with Given Invariants}
\author{Sebastian Pauli}
\author{Brian Sinclair}
\begin{abstract}
We give an algorithm that constructs 
a minimal set of polynomials defining all extension of
a $(\pi)$-adic field with given, inertia degree, ramification index, discriminant, 
ramification polygon, and 
residual polynomials of the segments of the ramification polygon.
\end{abstract}

\maketitle

\section{Introduction}

It follows from Krasner's Lemma that a local field has only finitely many extensions of a given degree and discriminant.  
Thus it is natural to ask whether one can generate a list of polynomials such that each extension is generated by exactly one of the polynomials.

For abelian extensions, local class field theory gives a one-to-one correspondence 
between the abelian extensions of $\KK$ and the open subgroups of the unit group $\KK^\times$ of $\KK$.
An algorithm that constructs the wildly ramified part of the class field as towers of extensions of degree $p$ was given in \cite{pauli-lcf}. 
Recently Monge \cite{monge} has published an algorithm that, given a subgroup of $\KK^\times$ of finite index, directly constructs the generating polynomial of the corresponding totally ramified extension.

In the non-abelian case, such a complete description is not yet known.  However, a description of all tamely ramified extensions is well known and all extensions of degree $p$ have been described completely by Amano \cite{amano}.  
Krasner \cite{krasner} gave a formula for the number of totally ramified extensions, using his famous lemma as a main tool.  
Following his approach, Pauli and Roblot \cite{pauli-roblot} presented an algorithm that returned a set of generating polynomials for all extensions of a given degree and discriminant.
They used the root-finding algorithm described by Panayi \cite{panayi} to obtain one generating polynomial for each extension.
A new approach for determining whether two polynomials generate the same extension was recently presented by Monge \cite{monge}.  
He introduces \emph{reduced polynomials} that yield a canonical set of generators for totally ramified extensions of $\KK$.   

Monge's methods also considerably reduce the number of generating polynomials that need to be considered when computing a set of polynomials defining all totally ramified extensions of $\KK$.  
We present an algorithm that for each extension with given invariants constructs a considerably smaller set of defining polynomials than 
the set obtained with Krasner's bound.
In many cases this eliminates the need to check whether two polynomials generate the same extension.
The polynomials constructed are reduced in Monge's sense.

\subsection*{Overview}
{In the first three sections of the paper, we examine extension invariants and how specifying each
invariant reduces the number of polynomials to be considered.}
We recall some of Krasner's results \cite{krasner} that are based on degree and discriminant
(Section \ref{sec disc}) and then add the ramification polygon as an additional invariant 
(Section \ref{sec ram pol}).
In Section \ref{sec res seg} we introduce an invariant based on the residual polynomials 
of the ramification polygon which consists of a polynomial over the residue class field
for each segment of the ramification polygon.
The residual polynomials together with ideas of Monge \cite{monge} reduce the 
number of polynomials to be considered considerably (Section \ref{sec res comp}).
In Section \ref{sec gen} we give an algorithm that uses the results of the previous sections to 
return a set of polynomials that generate all extensions 
with given invariants.  
In many cases this set contains exactly one polynomial for each extension.
Section \ref{sec ex} contains examples and comparisons with the implementations of the
algorithm by Pauli and Roblot \cite{pauli-roblot}.

\subsection*{Notation}

By convention fractions denoted $\slopenum/\slopeden$ or $\slopenum_i/\slopeden_i$ are always taken to be in lowest terms.

We denote by $\Qp$ the field of $p$-adic numbers and by $v_p$ the (exponential) valuation normalized such that  $v_p(p)=1$.
By $\KK$ we denote a finite extension of $\Qp$, by $\OK$ the valuation ring of $\KK$, and by $\pi$ a uniformizer of $\OK$.
We write $v_\pi$ for the valuation of $\KK$ that is normalized such that $v_\pi(\pi)=1$ and also denote the unique extension of $v_\pi$ to an algebraic closure $\clK$ of $\KK$ (or to any intermediate field) by $v_\pi$.

For $\gamma\in\clK^\times$ and $\delta\in\clK^\times$ we write
\( \gamma\sim\delta \)
if
\[ v(\gamma-\delta)>v(\gamma)\]
and make the supplementary assumption $0\sim0$.

For $\gamma\in\OK$ we denote by $\underline\gamma$ the class $\gamma+(\pi)$ in $\RK=\OK/(\pi)$,
by $\reps{\RK}$  a fixed set of representatives of $\RK$ in $\OK$,
and by $\repss{\RK}$ the set $\reps{\RK}$ without the representative for $\underline0\in\RK$.
For a polynomial $\varphi\in\OK[x]$ of degree $n$ we denote its coefficients by $\varphi_i$ ($0\le i\le n$)
such that $\varphi(x)=\varphi_n x^n+\varphi_{n-1}x^{n-1}+\dots+\varphi_0$ and write 
$ \varphi_i = \sum_{j=0}^\infty \varphi_{i,j} \pi^j. $ where $\varphi_{i,j}\in\reps{\RK}$.
If $\varphi$ is Eisenstein then $\varphi_n=1$, $\varphi_{0,1}\ne 0$ and $\varphi_{i,0}>0$ for $1\le i<n$.

In examples we use a table to represent sets of polynomials. 
Each cell contains a set from which the corresponding
coefficient $\varphi_{i,j}$ of the $\pi$-adic expansion of the coefficient $\varphi_i = \sum_{j=0}^\infty \varphi_{i,j} \pi^j$
of the polynomial $\varphi(x)=\varphi_n x^n+\varphi_{n-1}x^{n-1}+\dots+\varphi_0$ can be chosen.

\begin{example}
The Eisenstein polynomials of degree $n$ over $\OK$ are represented by the template:
\begin{center}
\begin{tabular}{l|cccccccccc}
             &$x^n$  &$x^{n-1}$ &$x^{n-2}$ &$\cdots$  &$x^4$ &$x^3$ &$x^2$ &$x^1$ &$x^0$ \\\hline
      \vdots &\vdots &\vdots    &\vdots    &          &\vdots&\vdots&\vdots&\vdots&\vdots\\
      $\pi^2$&$\{0\}$&$\reps{\RK}$    & $\reps{\RK}$     &$\cdots$  & $\reps{\RK}$ &$\reps{\RK}$  &$\reps{\RK}$  &$\reps{\RK}$  &$\reps{\RK}$  \\
      $\pi^1$&$\{0\}$&$\reps{\RK}$    & $\reps{\RK}$     &$\cdots$  & $\reps{\RK}$ &$\reps{\RK}$  &$\reps{\RK}$  &$\reps{\RK}$  &$\repss{\RK}$ \\
      $\pi^0$ &$\{1\}$&$\{0\}$  &$\{0\}$ &$\cdots$ &$\{0\}$ &$\{0\}$ &$\{0\}$ &$\{0\}$ &$\{0\}$ 
\end{tabular}
\end{center}
\end{example}


\section{Discriminant}\label{sec disc}

We recall some of the results Krasner used to obtain his formula for the number of extensions of a $p$-adic field \cite{krasner}.
These can also be found in \cite{pauli-roblot}.

The possible discriminants of finite extensions are given by Ore's conditions \cite{ore-bemerkungen}:

\begin{proposition}[Ore's conditions]\label{prop.ore}
Let $\KK$ be a finite extension of $\Qp$, $\OK$ its valuation ring with maximal ideal
$(\pi)$. Given $J_0 \in \Z$ let $a_0,b_0\in\Z$ be such that $J_0 = a_0n + b_0$ and
$0 \leq b_0 < n$. Then there exist totally ramified extensions
$\KL/\KK$ of degree $n$ and discriminant $(\pi)^{n+J_0-1}$ if and only if
\[
\min\{v_\pi(b_0) n, v_\pi(n) n\} \leq J_0 \leq v_\pi(n) n.
\]
\end{proposition}

The proof of Ore's conditions yields a certain form for the generating polynomials 
of extensions with given discriminant.

\begin{lemma}\label{lem ore coeff}
An Eisenstein polynomial $\varphi\in\OK[x]$ 
with discriminant $(\pi)^{n+J_0-1}$ where $J_0=a_0n+b_0$ with $0\le b_0<n$ fulfills Ore's conditions
if and only if 
\begin{align*}
v_\pi(\varphi_i) &\ge \max\{ 2 + a_0 - v_\pi(i), 1\}  \mbox{ for $0 < i < b_0$}, \\
v_\pi(\varphi_{b_0}) & = \max\{ 1 + a_0 - v_\pi(b_0), 1\}, \\
v_\pi(\varphi_i) & \ge \max\{ 1 + a_0 - v_\pi(i), 1\}  \mbox{ for $b_0 < i < n$}. 
\end{align*}
\end{lemma}

Krasner's Lemma yields a bound over which the coefficients of the $\pi$-adic expansion of the coefficients 
of a generating polynomial can be chosen to be $0$ \cite{krasner}.

\begin{lemma}\label{lem krasner bound}
Each totally ramified extension of degree $n$ with discriminant $(\pi)^{n+J_0-1}$
where $J_0=a_0n+b_0$ with $0\le b_0<n$ can be generated by an Eisenstein polynomial 
$\varphi\in\OK[x]$ with
$\varphi_{i,j}=0$ for $0\le i < n$ and $j>1 + 2a_0 + \frac{2b_0}{n}$.
\end{lemma}

With Lemma \ref{lem ore coeff} and Lemma \ref{lem krasner bound} we obtain a finite set of polynomials that generate all extensions 
of a given degree and discriminant.  In \cite{pauli-roblot} this set in conjunction with Krasner's mass formula \cite{krasner} 
and Panayi's root finding algorithm is used to obtain a generating polynomial for each extension of a given degree and discriminant.

\begin{example}\label{ex J10}
We want to find generating polynomials for all totally ramified 
extensions $\KL$ of $\Q_3$ of degree 9 with $v_3(\disc(\KL))=18$.
Denote by $\varphi=\sum_{i=0}^9\varphi_i x^i$ an Eisenstein polynomial generating such a field $\KL$.
By Lemma \ref{lem ore coeff} with $J_0=10$, $a_0=1$, and $b_0=1$ we get
\nrt{$v_\pi(\varphi_1)  = 2$} and
\ngt{$v_\pi(\varphi_i)  = 2 - v_\pi(i)$} for \ngt{$1 < i < n$}.
Furthermore by Lemma \ref{lem krasner bound} 
\nbt{$\varphi_{i,j}=0$ for $0\le i\le 9$ and $j>3$}.
Thus the template for the polynomials $\varphi$ is:
\begin{center}
\small
\renewcommand\arraystretch{1.1}
\renewcommand\tabcolsep{2pt}
\begin{tabular}{l|cccccccccc}
             &$x^9$    &$x^8$      &$x^7$      &$x^6$      &$x^5$      &$x^4$      &$x^3$      &$x^2$      &$x^1$        &$x^0$ \\\hline
      $3^4$&$\{0\}$    &\nbe{\{0\}}&\nbe{\{0\}}&\nbe{\{0\}}&\nbe{\{0\}}&\nbe{\{0\}}&\nbe{\{0\}}&\nbe{\{0\}}&\nbe{\{0\}}  &\nbe{\{0\}}  \\
      $3^3$&$\{0\}$    &$\{0,1,2\}$&$\{0,1,2\}$&$\{0,1,2\}$&$\{0,1,2\}$&$\{0,1,2\}$&$\{0,1,2\}$&$\{0,1,2\}$&$\{0,1,2\}$  &$\{0,1,2\}$  \\
      $3^2$&$\{0\}$    &$\{0,1,2\}$&$\{0,1,2\}$&$\{0,1,2\}$&$\{0,1,2\}$&$\{0,1,2\}$&$\{0,1,2\}$&$\{0,1,2\}$&\nre{\{1,2\}}&$\{0,1,2\}$  \\
      $3^1$&$\{0\}$    &\nge{\{0\}}&\nge{\{0\}}&$\{0,1,2\}$&\nge{\{0\}}&\nge{\{0\}}&$\{0,1,2\}$&\nge{\{0\}}&\nre{\{0\}}  &$\{1,2\}$ \\
      $3^0$&$\{1\}$    &$\{0\}$    &$\{0\}$    &$\{0\}$    &$\{0\}$    &$\{0\}$    &$\{0\}$    &$\{0\}$    &$\{0\}$      &$\{0\}$
\end{tabular}
\end{center}
\end{example}


\section{Ramification Polygons}\label{sec ram pol}

To distinguish totally ramified extensions further we use 
an additional invariant, namely the {ramification polygon}.  

\begin{definition}\label{def ram pol}
Assume that the Eisenstein polynomial $\varphi$ defines $\KL/\KK$.  The \emph{ramification polygon} $\rpol_\varphi$ of $\varphi$ is the Newton polygon $\mathcal{N}$ of the \emph{ramification polynomial} $\rho(x)=\varphi(\alpha x + \alpha)/(\alpha^n)\in  K(\alpha)[x]$ of $\varphi$, where $\alpha$ is a root of $\varphi$.
\end{definition}
The ramification polygon $\rpol_\varphi$ of $\varphi$ is an invariant of $\KL/\KK$ (see \cite[Proposition 4.4]{greve-pauli} for example) called the ramification polygon of $\KL/\KK$ denoted by $\rpol_{\KL/\KK}$.
Ramification polygons have been used to study ramification groups and reciprocity \cite{scherk},
compute splitting fields and Galois groups \cite{greve-pauli},
describe maximal abelian extensions \cite{lubin},
and answer questions of commutativity in $p$-adic dynamical systems \cite{li}.

Let $\varphi(x)=\sum_{i=0}^{n}\varphi_ix^i\in\KK[x]$ be an Eisenstein polynomial, denote by $\alpha$ a root of $\varphi$, 
and set $\KL=\KK(\alpha)$. Let $\rho(x)=\sum_{i=0}^{n}\rho_i x^i\in\KL[x]$ be the ramification polynomial of $\varphi$.
Then the coefficients of $\rho$ are
\[
\rho_i=\sum_{k=i}^n\binom{k}{i}\; \varphi_k\; \alpha^{k-n} 
\]
As $v_\alpha(\alpha)=1$ and $v_\alpha(\varphi_i)\in n\Z$ we obtain
\begin{equation}\label{eq val coeff rho}
v_\alpha(\rho_i) = \min_{i\leq k \leq n} \left\lbrace v_\alpha \left( \binom{k}{i}\; \varphi_k\; \alpha^k \right) - n\right\rbrace
                    = \min_{i\leq k\leq n}\left\lbrace n \left[ v_\pi \left( \binom{k}{i}\; \varphi_k \right) - 1\right] + k \right\rbrace.
\end{equation}

\begin{lemma}[{\cite[Lemma 1]{scherk}}] \label{lem scherk}
    Let $\varphi(x)=\sum_{i=0}^{n}\varphi_ix^i\in
  \KK[x]$ be an Eisenstein polynomial and $n=e_0 p^m$ with $p \nmid e_0$.
  Denote by $\alpha$ a root of $\varphi$ and set
  $\KL=\KK(\alpha)$. Then the following hold for the coefficients of the ramification polynomial
  $\rho(x)=\sum_{i=0}^{n}\rho_i x^i=\varphi(\alpha x+\alpha)/\alpha^n\in\OL[x]$ of $\varphi$:
  \begin{enumerate}
    \item $v_\alpha(\rho_i)\ge 0$ for all $i$;
    \item $v_\alpha(\rho_{p^m})=v_\alpha(\rho_n)=0$;
    \item $v_\alpha(\rho_i)\ge v_\alpha(\rho_{p^s})\mbox{ for }p^s\le
             i<p^{s+1}\mbox{ and }s<m$.
  \end{enumerate}
\end{lemma}

This gives the typical shape of the ramification polygon (see Figure \ref{fig ram pol}).

\begin{remark}
Throughout this paper we describe ramification polygons by the set of points 
\[\ppol=\rpolypoints\] where not all points in $\ppol$ have to be vertices 
of the polygon $\rpol$.
We write $\rpol=\ppol$.
This gives a finer distinction between fields by their ramification polygons and also 
allows for an easier description of the invariant based on the residual 
polynomials of the segments of the ramification polygon, see Section \ref{sec res seg}.
\end{remark}

\begin{figure}
\begin{center}
\begin{tikzpicture}[scale=0.6]
\draw[thick] (-1,0) -- (5,0);
\draw[thick, loosely dashed] (5,0) -- (8,0);
\draw[thick,->] (8,0) -- (20.5,0) node[anchor=west] {$i$};
\draw[thick,->] (-1,0) -- (-1,10.5) node[anchor=south] {$v_\alpha(\rho_i)$};

\draw[thick] (-0.5,10) -- (1,6) -- (4.9,2.75) ;
\draw[thick] (10,1) -- (15.5,0) -- (20,0);
\filldraw[black] (-0.5,10) circle (3pt)
                 (1,6) circle (3pt)
                 (2.75,4.5) circle (3pt)
                 (4.9,2.75) circle (3pt)
                 (10,1) circle (3pt)
                 (15.5,0) circle (3pt)
                 (20,0) circle (3pt);
\filldraw[black] (5.75,2.4) circle (1pt)
                 (6.7,2.05) circle (1pt)
                 (7.75,1.7) circle (1pt)
                 (8.8,1.35) circle (1pt);
\filldraw[black] (17,0) circle (3pt)
                 (18.5,0) circle (3pt);

\foreach \x/\xtext in {-0.5/1, 0.95/p^{s_1}\!\!, 2.75/p^{s_2}\!\!, 4.9/p^{s_3}\!\!, 10/p^{s_{u-1}}\!\!, 15.5/p^{s_u}=p^{v_p(n)}\!\!, 20/n}
  \draw[thick] (\x,0) -- (\x,-0.2) node[anchor=north] {$\xtext$};

\draw (0.5,8)     node[anchor=west] {$-\lambda_1$}
      (2.75,4.75) node[anchor=west] {$-\lambda_2$}
      (11.5,1.25) node[anchor=west] {$-\lambda_\ell$}
      (-0.5,10)   node[anchor=west] {$(1,J_0)$}
      (1,6.25)    node[anchor=west] {$(p^{s_1},J_1)$}
      (3.0,4)     node[anchor=east] {$(p^{s_2},J_2)$}
      (4.9,3)     node[anchor=west] {$(p^{s_3},J_3)$}
      (10,1)      node[anchor=east] {$(p^{s_{u-1}},J_{u-1})$}
      (14.5,1)    node[anchor=west] {$(p^{s_u},0)$}
      (20,1)      node[anchor=center] {$(n,0)$} ;
\end{tikzpicture}
\end{center}
\caption{Ramification polygon of an Eisenstein polynomial $\varphi$ of degree $n$ and discriminant $(\pi)^{n+J_0-1}$
with $\ell+1$ segments and $u-1$ points on the polygon with ordinate above $0$.}\label{fig ram pol}
\end{figure}
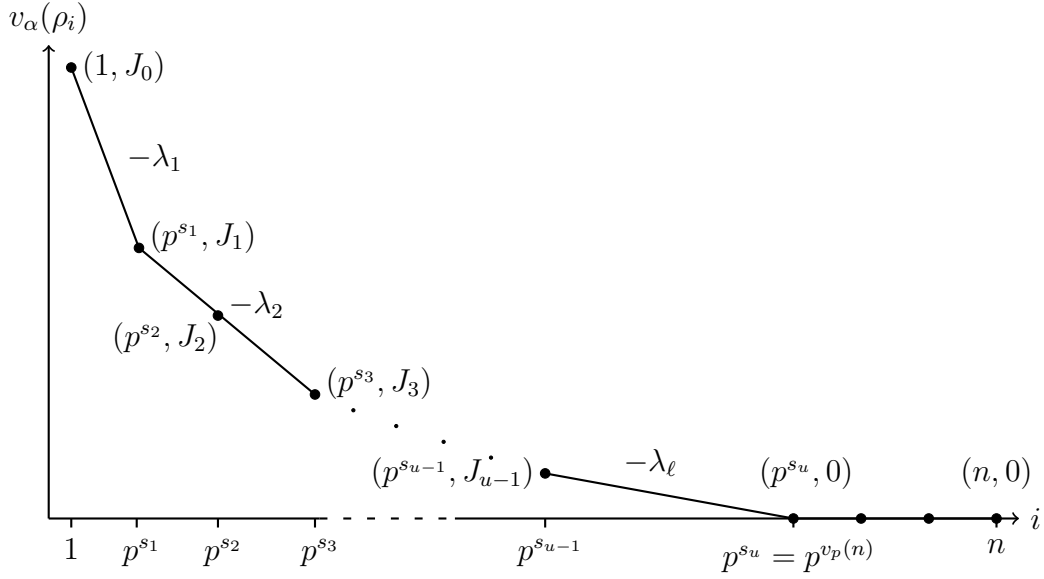

We now investigate the points on a ramification polygon further.

\begin{lemma}\label{lem rho sim}
Let $\rho=\sum_{i=1}^n\rho_i x^i$ be the ramification polynomial of an Eisenstein polynomial
$\varphi(x)=\sum_{i=0}^{n}\varphi_ix^i\in\OK[x]$.
Denote by
\[
\rpolypoints\subseteq \{(i,v_\alpha(\rho_i)):1\le i \le n\}
\]
the points on the ramification polygon of $\varphi$
and write $J_i=a_i n+b_i$ with $0\le b_i< n$.
\begin{enumerate}
\item {For $p^{s_u}\le i \le n$ we have $v_\alpha(\rho_i)=0$ and $\rho_i\equiv \binom{n}{i} \bmod (\alpha)$ if and only if $v_\alpha\binom{n}{i}=0$.}
\item For $0\le i\le u$ we have
\[
\rho_{p^{s_i}}\sim \varphi_{b_i}\binom{b_i}{p^{s_i}}\alpha^{b_i-n}.
\]
\end{enumerate}
\end{lemma}

It follows from (a) that, modulo $(\alpha)$,
the coefficients of the ramification polynomial that correspond to the horizontal segment of its Newton polygon
only depend on the degree of $\varphi$. 

\begin{lemma}\label{lem vphi}
If the ramification polygon of an Eisenstein polynomial
$\varphi\in\OK[x]$ has the points
$\rpolypoints$
where $J_i = a_in+b_i$ with $0\leq b_i\leq n-1$.
Then for $0\leq t\leq u$, we have
\[v_\pi(\varphi_i) \geq \left\{\begin{array}{ll}
  2+a_t-v_\pi \binom{i}{p^{s_t}} \mbox{ for } p^{s_t} \leq i < b_t\\[3\jot]
  1+a_t-v_\pi \binom{i}{p^{s_t}} \mbox{ for } b_t \leq i \leq n-1\\
  \end{array}\right. \]
and $v_\pi(\varphi_{b_t})=a_t+1-v_\pi \binom{b_t}{p^{s_t}}$ if $b_t\neq 0$.
\end{lemma}
\begin{proof}
By Equation (\ref{eq val coeff rho}), for all $k$ with $s_t\leq k \leq n$,
\[ J_t = a_tn+b_t \leq n \left[ v_\pi \left( \binom{k}{p^{s_t}}\; \varphi_k \right) - 1\right] + k,\]
which solved for $v_\pi(\varphi_k)$ gives
\[ 
1 + a_t - v_\pi \binom{k}{p^{s_t}} + \frac{b_t-k}{n} \leq v_\pi(\varphi_k) \mbox{ for }s_t\le k\le n.
\]
As $v_\pi(\varphi_k)$ is an integer, we may take the ceiling of the fraction.
As $0\leq b_t \leq n-1$ and $p^{s_t}\leq k \leq n$, if $k<b_t$, then $\left\lceil \frac{b_t-k}{n} \right\rceil=1$, and if $k\geq b_t$, then $\left\lceil \frac{b_t-k}{n} \right\rceil=0$. Therefore,
\[v_\pi(\varphi_i) \geq \left\{\begin{array}{ll}
  2+a_t-v_\pi \binom{i}{p^{s_t}} \mbox{ for } p^{s_t} \leq i < b_t\\[3\jot]
  1+a_t-v_\pi \binom{i}{p^{s_t}} \mbox{ for } b_t \leq i \leq n-1\\
  \end{array}\right..\]

Now if we consider a point $(p^{s_t},a_tn+b_t)$ with $b_t\neq 0$, then by Equation (\ref{eq val coeff rho}) we have
\[
a_t n+ b_t= \min_{p^{s_t}\leq k\leq n}\left\lbrace n \left[ v_\pi \left( \binom{k}{p^{s_t}}\; \varphi_k \right) - 1\right] + k \right\rbrace,
\]
and as $0<b_t<n$, the minimum is attained at $k=b_t$.  Hence
$a_t=\left[ v_\pi \left( \binom{b_t}{p^{s_t}}\; \varphi_{b_t} \right) - 1\right]$
and $v_\pi(\varphi_{b_t})=a_t+1-v_\pi \binom{b_t}{p^{s_t}}$.  
\end{proof}

From this, we can generalize Ore's conditions (Proposition \ref{prop.ore}) 
from a statement about the exponent of the discriminant, which is related 
to the ordinate of the point above 1, to the ordinates of all points.

\begin{lemma}\label{lem ore gen}
Let $\rpol_\varphi$ be the ramification polygon of $\varphi$ as in Lemma \ref{lem vphi}.
Then for each point $(p^{s_i},J_i)$ where $J_i = a_in+b_i$ with $0\leq b_i\leq n-1$,
\[\min \left\lbrace v_\pi\binom{b_i}{p^{s_i}} n, v_\pi \binom{n}{p^{s_i}} n\right\rbrace \leq J_i \leq v_\pi \binom{n}{p^{s_i}}n.\]
\end{lemma}
\begin{proof}
The $k=n$ term of Equation (\ref{eq val coeff rho}) is
\[
J_i\leq n \left[ v_\pi \left( \tbinom{n}{p^{s_i}} \varphi_n \right) -1 \right] +n = v_\pi \tbinom{n}{p^{s_i}}n.
\]
If $b_i\neq 0$, then by Lemma \ref{lem vphi},
$v_\pi(\varphi_{b_i})=a_i+1-v_\pi \binom{b_i}{p^{s_i}}$.
So $nv_\pi(\varphi_{b_i})+b_i=na_i+n-nv_\pi \binom{b_i}{p^{s_i}}+b_i$
and $nv_\pi(\varphi_{b_i})+b_i-n+nv_\pi \binom{b_i}{p^{s_i}} = na_i+b_i = J_i$.
As $\varphi$ is Eisenstein we have $v_\pi(\varphi_{b_i})\geq 1$, hence $nv_\pi(\varphi_{b_i})-n\geq 0$.
This combined with $b_i > 0$ gives us that
\[
J_i = nv_\pi(\varphi_{b_i})+b_i-n+nv_\pi \binom{b_i}{p^{s_i}}
 \geq b_i+nv_\pi \binom{b_i}{p^{s_i}}
 \geq nv_\pi \binom{b_i}{p^{s_i}}.
\]
If $b_i=0$, then the minimum term of Equation (\ref{eq val coeff rho}) defining $J_i$ must be such that $k|n$,
which only occurs in the $k=n$ term, so $J_i=v_\pi \binom{n}{p^{s_i}}n$,
which is less than $v_\pi \binom{0}{p^{s_i}}n = \infty$.
\end{proof}

\begin{lemma}\label{lem vphi no vert}
Let $\rpol_\varphi$ be the ramification polygon of an Eisenstein polynomial $\varphi\in\OK[x]$ with points
\[\rpol_\varphi = \rpolypoints,\]
but no point with abscissa $p^i$, where $s_t < i < s_{t+1}$ for some $1\le t\le u$. Then for $k$ such that $p^i\leq k\leq n$,
\[
v_\pi(\varphi_k) > \frac{1}{n}\left[ \frac{J_{t+1}-J_t}{p^{s_{t+1}}-p^{s_t}} (p^i-p^{s_t})+J_t-k \right]+1-v_\pi\binom{k}{p^i}
\]
\end{lemma}
\begin{proof}
{If there is no point on $\rpol_\varphi$ with abscissa $p^i$, then the point $(p^i,v_\alpha(\rho_{p^i}))$ must be above the segment from $(p^{s_t},J_t)$ to $(p^{s_{t+1}},J_{t+1})$.}
Thus,
$
\frac{J_{t+1}-J_t}{p^{s_{t+1}}-p^{s_t}} (p^i-p^{s_t})+J_t < v_\alpha(\rho_{p^i}),
$
and so by Equation (\ref{eq val coeff rho}), for $k$ in $p^i\leq k\leq n$,
\[
\frac{J_{t+1}-J_t}{p^{s_{t+1}}-p^{s_t}} (p^i-p^{s_t})+J_t
                    < n \left[ v_\pi \left( \binom{k}{p^i}\; \varphi_k \right) - 1\right] + k .
\]
Solving for $v_\pi(\varphi_k)$ provides the result of the lemma.
\end{proof}

We collect the results of Lemmas \ref{lem vphi} and \ref{lem vphi no vert} to define functions 
$l_{\rpol_\varphi}(i,s)$ for $1\leq s\leq s_u$ and $p^s\leq i\leq n$ that give the minimum valuation of $\varphi_i$
due to a point (or lack thereof) above $p^s$ on the ramification polygon $\rpol_\varphi$ of $\varphi$.
By taking the maximum of these over all $s$, we define $L_{\rpol_\varphi}(i)$ so that
$v_\pi(\varphi_i)\ge L_{\rpol_\varphi}(i)$ for $1\le i \le n-1$.

\begin{definition}\label{def ells}
Let $\rpol_\varphi$ be the ramification polygon of $\varphi$ with points
\[ \rpol_\varphi = \rpolypoints, \]
and where $J_i = a_in+b_i$ with $0\leq b_i\leq n-1$. For $0\leq t\leq u$, let
\[
l_{\rpol_\varphi}(i,s_t)=\left\{\begin{array}{ll}
\max\{ 2 + a_t - v_\pi \binom{i}{p^{s_t}}, 1\} & \mbox{ if } p^{s_t} \leq i < b_t, \\[3\jot]
\max\{ 1 + a_t - v_\pi \binom{i}{p^{s_t}}, 1\} & \mbox{ if } i \geq b_t. \\
\end{array}\right. 
\]
If there is no point above $p^w$ with $s_t<w<s_{t+1}$, then for $p^w\leq i\leq n-1$, let
\[
l_{\rpol_\varphi}(i,w) = \max\left\lbrace
\left\lceil \frac{1}{n}\left[ \frac{J_{t+1}-J_t}{p^{s_{t+1}}-p^{s_t}} (p^w-p^{s_t})+J_t-k \right]+1-v_\pi\binom{k}{p^w} \right\rceil
,1\right\rbrace
\]
Finally, set
\[
L_{\rpol_\varphi}(i) = 
\left\{\begin{array}{ll}
1 &\mbox{if } i=0\\
\max \{l_{\rpol_\varphi}(i,t) : p^t \leq i\} &\mbox{if }1\leq i\leq n-1\\
0 &\mbox{if } i=n
\end{array}\right..
\]
\end{definition}

\begin{lemma}\label{lem ps div js}
Let $\rpol_\varphi$ be the ramification polygon of $\varphi$ with points
\[\rpol_\varphi = \rpolypoints \]
where $J_i = a_in+b_i$ with $0\leq b_i\leq n-1$.  Then $p^{s_i}\mid J_i$ for $0\leq i \leq u$.
\end{lemma}
\begin{proof}
As $J_0$ is an integer, $p^0=1$ divides $J_0$, and as $J_u=0$, clearly $p^{s_u} | J_u$.

Suppose that for some $1\leq i < u$ we have $v_p(J_i)=t < s_i$.  
If $\rpol$ is the ramification polygon of $\varphi$ with ramification polynomial $\rho$ and contains $(p^{s_i},J_i)$, then $t < s_i$ must imply that $J_i < v_\alpha(\rho_{p^t})$, which is bounded above by the $k=b_i$ term of Equation (\ref{eq val coeff rho}).
By Lemma \ref{lem vphi}, we have that $v_\pi(\varphi_{b_i})=a_i+1-v_\pi \binom{b_i}{p^{s_i}}$.
If we substitute this value of $v_\pi(\varphi_{b_i})$ into Equation (\ref{eq val coeff rho}), then
\[
v_\alpha(\rho_{p^t})
  \leq    n \left[ v_\pi \binom{{b_i}}{p^t} + v_\pi(\varphi_{b_i}) - 1 \right] + {b_i}\\
  =    n \left[ v_\pi \binom{{b_i}}{p^t} + a_i - v_\pi \binom{{b_i}}{p^{s_i}} \right] + {b_i}
\]
As $p^t||b_i$, the $p^t$-term of the base $p$ expansion of $b_i$ is non-zero, so
$v_p \binom{{b_i}}{p^t}=0$ and consequently $v_\pi \binom{{b_i}}{p^t}=0$.
Thus,
$ v_\alpha(\rho_{p^t}) \leq n \left[ a_i - v_\pi \binom{{b_i}}{p^i} \right] + {b_i} \leq a_in+b_i = J_i. $
This implies that $\rpol$ cannot have the point $(p^{s_i},J_i)$, and by contradiction, our claim is shown.
\end{proof}

So far we have described many necessary conditions for ramification polygons.
We now propose a necessary and sufficient description of a ramification polygon of an extension.

\begin{proposition}\label{prop ram pol iff}
Let
\[\ppol = \rpolypoints,\]
be a convex polygon with points
where $J_i = a_in+b_i$ with $0\leq b_i\leq n-1$.
There is an extension $\KL/\KK$ with ramification polygon $\ppol$, if and only if
\begin{enumerate}
\item \label{prop ram pol iff ore}
For each $J_i$, $\min \left\lbrace v_\pi\binom{b_i}{p^{s_i}} n, v_\pi \binom{n}{p^{s_i}} n\right\rbrace \leq J_i \leq v_\pi \binom{n}{p^{s_i}}n.$
\item \label{prop ram pol iff sameb}
If $b_i = b_k$, then $a_i = a_k - v_\pi \binom{b}{p^{s_k}} + v_\pi \binom{b}{p^{s_i}}$ where $b_i=b_k$.
\item \label{prop ram pol iff val}
For each point $(p^{s_i},a_in+b_i)$, we have that
\[ a_i \geq \left\lbrace \begin{array}{rl}
  1 + a_t - v_\pi \binom{b_i}{p^{s_t}} + \binom{b_i}{p^{s_i}} &\mbox{ if } p^{s_t} \leq b_i < b_t\\[3\jot]
      a_t - v_\pi \binom{b_i}{p^{s_t}} + \binom{b_i}{p^{s_i}} &\mbox{ if } b_i \geq b_t\\
\end{array}\right. \]
for all other points $(p^{s_t},J_t)$ with $J_t=a_tn+b_t\neq 0$.
\item \label{prop ram pol iff nopoint}
If there is no point of $\ppol$ above $p^\npl$, with $s_t<\npl <s_{t+1}$, then for each point $(p^{s_k},a_kn+b_k)$ 
of $\ppol$ with $b_k>p^\npl$,
\[
a_k > \frac{1}{n} \left[ \frac{J_{t+1}-J_t}{p^{s_{t+1}}-p^{s_t}} (p^\npl-p^{s_t})+ J_t - b_k \right] - v_\pi \binom{b_k}{p^\npl} + v_\pi \binom{b_k}{p^{s_k}}.
\]
\item \label{prop ram pol iff tamepts}
The points with abscissa greater than $p^{s_u}$ are $(i,0)$ where $v_\pi \binom{n}{i}=0$.
\end{enumerate}
\end{proposition}
\begin{proof}
Suppose $\ppol$ is the ramification polygon for $\KL/\KK$ with generating Eisenstein polynomial $\varphi$.
Assumption \ref{prop ram pol iff ore} follows from Corollary \ref{lem ore gen}.
If $b_i = b_k$, then by Lemma \ref{lem vphi}
\[
v_\pi (\varphi_{b_i}) = a_i + 1 - v_\pi \tbinom{b_i}{p^{s_i}} = a_k + 1 - v_\pi \tbinom{b_i}{p^{s_k}}.
\]
Thus  $a_i = a_k - v_\pi \binom{b_i}{p^{s_k}} + v_\pi \binom{b_i}{p^{s_i}}$,
giving us assumption \ref{prop ram pol iff sameb}.
Let $(p^{s_i},a_in+b_i)$ be a point of $\ppol$, then
by Lemma \ref{lem vphi}, we have that for all other points $(p^{s_t},J_t)$,
\[v_\pi(\varphi_{b_i}) =  a_i + 1 - v_\pi \binom{b_i}{p^{s_i}} \geq \left\{\begin{array}{ll}
  2+a_t-v_\pi \binom{b_i}{p^{s_t}} \mbox{ for } p^{s_t} \leq b_i < b_t\\[3\jot]
  1+a_t-v_\pi \binom{b_i}{p^{s_t}} \mbox{ for } b_i \geq b_t\\
  \end{array}\right.,\]
from which we see assumption \ref{prop ram pol iff val}.
If there no point of $\ppol$ above $p^\npl$, with $s_t < \npl < s_{t+1}$, 
then by Lemma \ref{lem vphi no vert}, for each point $(p^{s_i},a_in+b_i)$ of $\ppol$ with $b_i>p^\npl$,
\[
v_\pi(\varphi_{b_i}) =  a_i + 1 - v_\pi \binom{b_i}{p^{s_i}} \geq
\frac{1}{n}\left[ \frac{J_{t+1}-J_t}{p^{s_{t+1}}-p^{s_t}} (p^\npl-p^{s_t})+J_t-b_i \right]+1-v_\pi\binom{b_i}{p^\npl},
\]
from which we have assumption \ref{prop ram pol iff nopoint}.
Assumption \ref{prop ram pol iff tamepts} is given by Lemma \ref{lem rho sim}.
Thus, if $\ppol$ is a ramification polygon of an extension $\KL/\KK$, then these properties are necessary.

Next we will show sufficiency by constructing a polynomial $\psi(x)=\sum \psi_ix^i\in\OK[x]$ such that $\rpol_\psi=\ppol$.
First, we let $\psi_n=1$ and $\psi_0$ be an element of valuation 1 in $\OK$.
For each point $(p^{s_i},a_in+b_i)$ in $\ppol$, with $b_i\neq 0$, let $\psi_{b_i}$ be an element of $\OK$ with valuation $1+a_i-v_\pi\binom{b_i}{p^{s_i}}$.
By assumption \ref{prop ram pol iff sameb}, $\psi_{b_i}$ is well defined even if it is given by multiple points as those definitions coincide,
and by assumption \ref{prop ram pol iff ore} we have that $v_\pi(\psi_{b_i})\geq 1$.
If $\psi_j$ in $0<j<n$ is not assigned by some $b_i$, we set $\psi_j=0$.
We now have an  Eisenstein polynomial $\psi$, and we proceed by computing $\rpol_\psi$.

Let $\rpol_\psi$ be the ramification polygon of $\psi$, the Newton polygon $\mathcal{N}$ of the ramification polynomial $\rho(x)=\psi(\alpha x + \alpha)/(\alpha^n)\in  K(\alpha)[x]$, where $\alpha$ is a root of $\psi$.
Let $\rho(x)=\sum \rho_ix^i$.
Let $B$ be the set of nonzero $b_i$ in the points of $\ppol$.
For all $0<i<n$ with $i\notin B$, $v_\pi(\psi_i)=\infty$, so we can simplify Equation (\ref{eq val coeff rho}) by only needing to consider terms $k\in B\cup\{n\}$ to
\[v_\alpha(\rho_i) = \min \left\lbrace
  \min_{k\in B, k\geq i}\left\lbrace n \left[ v_\pi \left( \binom{k}{i}\; \psi_k \right) - 1\right] + k \right\rbrace,
  n  v_\pi \binom{k}{i}
  \right\rbrace.\]
Substitution of our values for $v_\pi(\psi_{b_t})$ gives
\[v_\alpha(\rho_i) = \min \left\lbrace
  \min_{\{
  (p^{s_k},J_k)\in \ppol : b_k \geq i
  \} }\left\lbrace n \left[ a_k - v_\pi\binom{b_k}{p^{s_k}} + v_\pi \binom{b_k}{i} \right] + b_k \right\rbrace,
  n  v_\pi \binom{n}{i}
  \right\rbrace.\]

Consider $(p^{s_i},a_in+b_i)\in\ppol$, and let us find $v_\alpha(\rho_{p^{s_i}})$.
\begin{equation}\label{eq rho bk}
v_\alpha(\rho_{p^{s_i}}) = \min \left\lbrace
  \min_{\{
  (p^{s_k},J_k)\in \ppol : b_k \geq p^{s_i}\}
  }\left\lbrace n \left[ a_k - v_\pi\binom{b_k}{p^{s_k}} + v_\pi \binom{b_k}{p^{s_i}} \right] + b_k \right\rbrace,
  n  v_\pi \binom{n}{p^{s_i}}
  \right\rbrace.\end{equation}
If $b_i\neq 0$, then the $b_k=b_i$ term in the minimum is $a_in+b_i$.  For $(p^{s_k},a_kn+b_k)\in\ppol$ with $p^{s_i} \leq b_k < b_i$, by assumption \ref{prop ram pol iff val}, we have $a_k \geq 1 + a_i - v_\pi \binom{b_k}{p^{s_i}} + \binom{b_k}{p^{s_k}}$. Thus, for all of the terms of (\ref{eq rho bk}) with $p^{s_i} \leq b_k < b_i$,
\[ n \left[ a_k - v_\pi\binom{b_k}{p^{s_k}} + v_\pi \binom{b_k}{p^{s_i}} \right] + b_k \geq 
n \left[ 1 + a_i \right] + b_k \geq a_in+b_i \]
For points $(p^{s_k},a_kn+b_k)$ on $\ppol$ with $b_k \geq b_i$, by assumption \ref{prop ram pol iff val}, we have $a_k \geq a_i - v_\pi \binom{b_k}{p^{s_i}} + \binom{b_k}{p^{s_k}}$. Thus, for all of the terms of Equation (\ref{eq rho bk}) with $b_k \geq b_i$,
\[ n \left[ a_k - v_\pi\binom{b_k}{p^{s_k}} + v_\pi \binom{b_k}{p^{s_i}} \right] + b_k \geq 
a_in  + b_k \geq a_in+b_i \]
Thus $v_\alpha(\rho_{p^{s_i}}) = \min \left\lbrace a_in+b_i,  n  v_\pi \binom{n}{p^{s_i}}\right\rbrace$, which is $a_in+b_i$ by assumption \ref{prop ram pol iff ore}.
On the other hand, if $b_i=0$, then $a_i = v_\pi \binom{n}{p^{s_i}}$, and for all of the terms of the inside minimum of 
Equation (\ref{eq rho bk}), as $a_k \geq a_i - v_\pi \binom{b_k}{p^{s_i}} + \binom{b_k}{p^{s_k}}$, we have
\[ n \left[ a_k - v_\pi\binom{b_k}{p^{s_k}} + v_\pi \binom{b_k}{p^{s_i}} \right] + b_k \geq 
a_in  + b_k \geq a_in = nv_\pi \binom{n}{p^{s_i}} \]
So, $v_\alpha(\rho_{p^{s_i}}) = a_in$, and all of the points of $\ppol$ are points of $\rpol_\psi$.

Suppose there is no point on $\ppol$ with abscissa $p^\npl$ for some $\npl$ with $s_t < \npl < s_{t+1}$. 
We take our assumption
\[
a_k > \frac{1}{n} \left[ \frac{J_{t+1}-J_t}{p^{s_{t+1}}-p^{s_t}} (p^\npl-p^{s_t})+ J_t - b_k \right] - v_\pi \binom{b_k}{p^\npl} + v_\pi \binom{b_k}{p^{s_k}},
\]
and substitute it into Equation (\ref{eq rho bk}).
After simplifying we  get
\[
v_\alpha(\rho_{p^i}) > \min \left\lbrace
  \min_{\{
  (p^{s_k},J_k)\in \ppol : b_k \geq p^{s_i}\} }
  \left\lbrace \frac{J_{t+1}-J_t}{p^{s_{t+1}}-p^{s_t}} (p^i-p^{s_t})+ J_t \right\rbrace,
  n  v_\pi \binom{n}{p^{s_i}} \right\rbrace. 
\]
As the $v_\alpha(\rho_{p^i})$ must be greater than the ordinate above $p^i$ on the line segment between $(p^{s_t},J_t)$ and $(p^{s_{t+1}},J_{t+1})$, there is no point on $\rpol_\psi$ with abscissa $p^i$.

Finally, by Lemma \ref{lem rho sim}, $\rpol_\psi$ has 
points satisfying Assumption \ref{prop ram pol iff tamepts}.  Thus $\rpol_\psi=\ppol$.
\end{proof}

\begin{proposition}\label{prop R iff}
An Eisenstein polynomial $\varphi$ has ramification polygon $\rpol$ with points
\[\rpol = \rpolypoints,\]
where $J_i = a_in+b_i$ with $0\leq b_i\leq n-1$,
if and only if
\begin{enumerate}
\item $v_\pi(\varphi_i) \geq L_\rpol(i)$
\item For $0\leq t\leq u$, $v_\pi(\varphi_{b_t})=L_\rpol(b_t)$ if $b_t\neq 0$.
\end{enumerate}
where $L_\rpol$ is as defined in Definition \ref{def ells}.
\end{proposition}
\begin{proof}
If $\varphi$ has ramification polygon $\rpol$, then this is the result
of Lemmas \ref{lem vphi} and \ref{lem vphi no vert}.

Suppose $\varphi$ satisfies these assumptions and $\rho$ is the ramification polynomial of $\varphi$.
If $(p^{s_t},J_t=a_tn+b_t)$ is a point of $\rpol$, then substitution of $l_\rpol(k,s_t)$
for $v_\pi(\varphi_k)$ into Equation (\ref{eq val coeff rho}) gives us
\[
v_\alpha(\rho_{p^{s_t}}) = \min \left\{
\min_{p^{s_t}\leq k < b_t} \{na_t+n+k\},
\min_{b_t \leq k < n} \{na_t+k\},
nv_\pi \binom{n}{p^{s_t}}
\right\} 
\]
If $b_t = 0$, then this reduces to
\[ v_\alpha(\rho_{p^{s_t}}) = \min \left\{
na_t+n+p^{s_t},
nv_\pi \binom{n}{p^{s_t}}
\right\} = nv_\pi \binom{n}{p^{s_t}} = J_t.
\]
as $na_t+n+p^{s_t} \geq J_t = nv_\pi \binom{n}{p^{s_t}}$, by Proposition \ref{prop ram pol iff} (a).
If $b_t \neq 0$, then this reduces to
\[ v_\alpha(\rho_{p^{s_t}}) = \min \left\{
na_t+b_t,
nv_\pi \binom{n}{p^{s_t}}
\right\}
= na_t+b_t = J_t
\]
as $J_t \leq nv_\pi \binom{n}{p^{s_t}}$, by Proposition \ref{prop ram pol iff} (a).
So $\rpol_\varphi$ contains the points of $\rpol$.

If there is no point on $\rpol$ with abscissa $p^i$, with $s_t < i < s_{t+1}$, then for $k$ in $p^i\leq k\leq n$,
\[
v_\pi(\varphi_k) \geq l_\rpol(k,i)
> \frac{1}{n}\left[ \frac{J_{t+1}-J_t}{p^{s_{t+1}}-p^{s_t}} (p^i-p^{s_t})+J_t-k \right]+1-v_\pi\binom{k}{p^i}.
\]
Some algebraic manipulation of this inequality gives us
\[
\frac{J_{t+1}-J_t}{p^{s_{t+1}}-p^{s_t}} (p^i-p^{s_t})+J_t
                    < n \left[ v_\pi \left( \binom{k}{p^i}\; \varphi_k \right) - 1\right] + k ,
\]
which shows that $v_\alpha(\rho_{p^i}) = \min_{p^i\leq k\leq n} \left\{ n \left[ v_\pi \left( \binom{k}{p^i}\; \varphi_k \right) - 1\right] + k\right\}$ is greater than the value above $p^i$ on the segment from $(p^{s_t},J_t)$ to $(p^{s_{t+1}},J_{t+1})$.
So there is no point on $\rpol_\varphi$ above $p^i$, and thus $\rpol_\varphi = \rpol$. 
\end{proof}

\begin{definition}\label{def hhf}
We call a polygon $\rpol$ with points
\[\rpol = \rpolypoints,\]
that fulfills the conditions of Proposition \ref{prop ram pol iff}
a \emph{ramification polygon}.
We call the function $\hhf_\rpol: \Rplus \to \Rplus$, 
$\lambda\mapsto \min_{0\leq i\leq u}\{\frac{1}{n}(J_i + \lambda p^{s_i}) \}$
the \emph{Hasse-Herbrand function} of $\rpol$.
\end{definition}

\begin{remark}
The function $\hhf_\rpol$ in Definition \ref{def hhf} agrees with the connections between the 
ramification polygon and the Hasse-Herbrand transition function as observed 
in \cite{lubin,li}.  Note that these works 
define the ramification polygon as the Newton polygon of $\varphi(x+\alpha)$. 
For normal extensions $\KL/\KK$, our function $\hhf_\rpol$ agrees with the 
classical $\hhf_{\KL/\KK}$ defined in \cite{serre,fv}.
For non-Galois extensions, our function agrees with the transition function
for ramification sets defined by Helou in \cite{helou}.
\end{remark}

\begin{example}[Example \ref{ex J10} continued]\label{ex R2}
There are three possible ramification polygons for extensions $\KL$ of 
$\Q_3$ of degree $9$ with $v_3(\disc(\KL))=18$.
These polygons are 
$\rpol_1=\{(1,10),(9,0)\}$,
$\rpol_2=\{(1,10),(3,3),(9,0)\}$, and
$\rpol_3=\{(1,10),(3,6),(9,0)\}$
and are illustrated in Figure \ref{fig ram pol ex}.

\begin{figure}
\begin{center}
\begin{tikzpicture}[scale=0.4]
\draw[thick,->] (0,0) -- (10,0) node[anchor=west] {$i$};
\draw[thick,->] (0,0) -- (0,11) node[anchor=south] {$v_\alpha(\rho_i)$};
\draw[thick] (1,10) -- (9,0);
\filldraw[black] (1,10) circle (3pt)
                 (9,0)  circle (3pt);
\draw (1,10)  node[anchor=west] {$(1,10)$}
      (8.2,1) node[anchor=west] {$(9,0)$};
\draw (4,6)   node[anchor=west] {$-\frac{5}{4}$};
\draw (-1,-1.5)   node[anchor=west] {$\rpol_1=\{(1,10),(9,0)\}$};

\draw[thick,->] (13,0) -- (23,0) node[anchor=west] {$i$};
\draw[thick,->] (13,0) -- (13,11) node[anchor=south] {$v_\alpha(\rho_i)$};
\draw[thick] (14,10) -- (16,3) -- (22,0);
\filldraw[black] (14,10) circle (3pt)
                 (16,3)  circle (3pt)
                 (22,0)  circle (3pt);
\draw (14,10)  node[anchor=west] {$(1,10)$}
      (17,3.5) node[anchor=west] {$(3,3)$}
      (22,1)   node[anchor=center] {$(9,0)$};
\draw (15,7)     node[anchor=west] {$-\frac{7}{2}$}
      (18.5,2.3) node[anchor=west] {$-\frac{1}{2}$};
\draw (12,-1.5)   node[anchor=west] {$\rpol_2=\{(1,10),(3,3),(9,0)\}$};

\draw[thick,->] (26,0) -- (36,0) node[anchor=west] {$i$};
\draw[thick,->] (26,0) -- (26,11) node[anchor=south] {$v_\alpha(\rho_i)$};
\draw[thick] (27,10) -- (29,6) -- (35,0);
\filldraw[black] (27,10) circle (3pt)
                 (29,6)  circle (3pt)
                 (35,0)  circle (3pt);
\draw (27,10)  node[anchor=west] {$(1,10)$}
      (29,6.5) node[anchor=west] {$(3,6)$}
      (34.2,1) node[anchor=west] {$(9,0)$};
\draw (28.2,8.1) node[anchor=west] {$-2$}
      (31.5,4.1) node[anchor=west] {$-1$};
\draw (25,-1.5)   node[anchor=west] {$\rpol_3=\{(1,10),(3,6),(9,0)\}$};
\end{tikzpicture}
\end{center}
\caption[\texttt{Possible ramification polygons of extensions 
$\mathtt{\KL}$ of $\mathtt{\Q_3}$ of degree 9 with $\mathtt{v_3(\disc(\KL))=18}$.}]
{Possible ramification polygons of extensions 
$\KL$ of $\Q_3$ of degree $9$ with $v_3(\disc(\KL))=18$.}
\label{fig ram pol ex}
\end{figure}
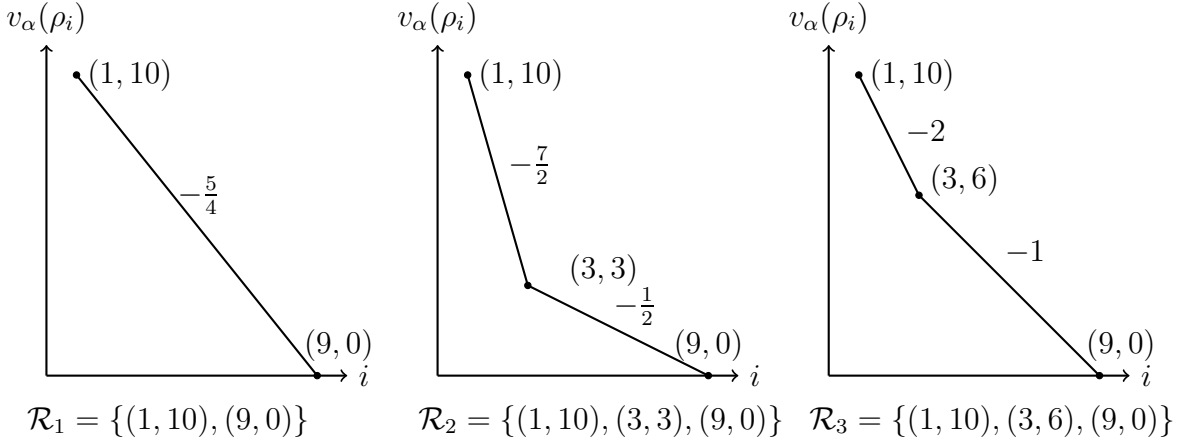

Since by Lemma \ref{lem vphi} we have \nrt{$v(\varphi_3)=1$}, 
the polynomials $\varphi$ generating extensions with ramification polygon $\rpol_2$ are given by:
\begin{center}
\small
\renewcommand\arraystretch{1.1}
\renewcommand\tabcolsep{2pt}
\begin{tabular}{l|cccccccccc}
             &$x^9$    &$x^8$      &$x^7$      &$x^6$      &$x^5$      &$x^4$      &$x^3$      &$x^2$      &$x^1$        &$x^0$ \\\hline
      $3^4$&$\{0\}$    &\ote{\{0\}}&\ote{\{0\}}&\ote{\{0\}}&\ote{\{0\}}&\ote{\{0\}}&\ote{\{0\}}&\ote{\{0\}}&\ote{\{0\}}  &\ote{\{0\}}  \\
      $3^3$&$\{0\}$    &$\{0,1,2\}$&$\{0,1,2\}$&$\{0,1,2\}$&$\{0,1,2\}$&$\{0,1,2\}$&$\{0,1,2\}$&$\{0,1,2\}$&$\{0,1,2\}$  &$\{0,1,2\}$  \\
      $3^2$&$\{0\}$    &$\{0,1,2\}$&$\{0,1,2\}$&$\{0,1,2\}$&$\{0,1,2\}$&$\{0,1,2\}$&$\{0,1,2\}$&$\{0,1,2\}$&\ote{\{1,2\}}&$\{0,1,2\}$  \\
      $3^1$&$\{0\}$    &\ote{\{0\}}&\ote{\{0\}}&$\{0,1,2\}$&\ote{\{0\}}&\ote{\{0\}}&\nre{\{1,2\}}&\ote{\{0\}}&\ote{\{0\}}  &$\{1,2\}$ \\
      $3^0$&$\{1\}$    &$\{0\}$    &$\{0\}$    &$\{0\}$    &$\{0\}$    &$\{0\}$    &$\{0\}$    &$\{0\}$    &$\{0\}$      &$\{0\}$
\end{tabular}
\end{center}
\end{example}


\section{Residual Polynomials of Segments}\label{sec res seg}

Residual (or associated) polynomials were introduced by Ore \cite{ore-newton}. 
They yield information about the unramified part of the extension generated by the factors of a polynomial. This makes them a useful tool in the computation of ideal decompositions and integral bases \cite{guardia-montes-nart,montes,montes-nart} and the closely related problem of polynomial factorization over local fields \cite{guardia-nart-pauli, pauli-pf2}.

\begin{definition}[Residual polynomial]\label{def res poly}
Let $\KL$ be a finite extension of $\KK$ with uniformizer $\alpha$.  
Let  $\rho(x)=\sum_i \rho_i x^i\in\OL[x]$.
Let $\spol$ be a segment of the Newton polygon of 
$\rho$  of length $l$ with
endpoints $(k,v_\alpha(\rho_k))$ and $(k+l,v_\alpha(\rho_{k+l}))$, and slope $-\slopenum/\slopeden=\left(v_\alpha(\rho_{k+l})-v_\alpha(\rho_k)\right)/l$
then
\[
\RA(x)=\sum_{j=0}^{l/e}\underline{\rho_{je+k}\alpha^{jh-v_\alpha(\rho_k)}}x^{j}\in\RK[x]
\]
is called the \emph{residual polynomial} of $\spol$.
\end{definition}

\begin{remark}
The ramification polygon of a polynomial $\varphi$ and the residual polynomials of its segments yield a subfield $\KM$ 
of the splitting field $\KN$ of 
$\varphi$, such  that $\KN/\KM$ is a $p$-extension \cite[Theorem 9.1]{greve-pauli}. 
\end{remark}

From the definition we obtain some of the properties of residual polynomials.

\begin{lemma}\label{lem res pol}
Let $\KL$ be a finite extension of $\KK$ with uniformizer $\alpha$.  Let $\rho\in\OL[x]$.
Let $\npol$ be the Newton polygon of $\rho$  with segments $\spol_1,\dots,\spol_\ell$ and let
$\RA_1,\dots,\RA_\ell$ be the corresponding residual polynomials.
\begin{enumerate}
\item \label{lem res pol int}
If $\spol_i$ 
has integral slope $-\slopenum\in\Z$
with endpoints $(k,v_\alpha(\rho_k))$ and $(k+l,v_\alpha(\rho_{k+l}))$ then
$
\RA_i(x)=\sum_{j=0}^{l}\underline{\rho_{j+k}\alpha^{j\slopenum-v_\alpha(\rho_k)}}\,x^{j}
=\underline{\rho(\alpha^\slopenum x)\alpha^{-k-v_\alpha(\rho_k)}\,x^{n-l}}
\in\RK[x].
$
\item \label{lem res pol end}
If for $1\le i\le\ell-1$ the leading coefficient of $\RA_i$ is denoted by $\RA_{i,\deg\RA_i}$ and
$\RA_{i+1,0}$ is the constant coefficient of $\RA_{i+1}$ then $\RA_{i,\deg\RA_i}=\RA_{i+1,0}$. 
\item If $\rho$ is monic then $\RA_\ell$ is monic.
\end{enumerate}
\end{lemma}

From now on we consider the residual polynomials of the segments of a ramification polygon.
From the definition of the residual polynomials and Lemma \ref{lem rho sim} we obtain:

\begin{proposition}\label{prop res pol seg}
Let $\varphi\in\OK[x]$ be Eisenstein of degree $n=p^r e_0$ with $\gcd(p,e_0)=1$, let $\alpha$ be a root of $\varphi$, 
$\rho$ the ramification polynomial, and $\rpol_\varphi$ the ramification polygon of $\varphi$.
\begin{enumerate}
\item
If $e_0\ne 1$ then $\rpol_\varphi$ has a horizontal
segment of length $p^r(e_0-1)$ with residual polynomial 
$\RA=\sum_{i=0}^{n-p^r}\RA_{i}x^i$ where 
$\RA_{i}=\underline{\binom{n}{i}}\ne\underline0$ if and only if  $v_\alpha\binom{n}{i}=0$.
\item If $(p^{s_k},J_k),\dots,(p^{s_l},J_l)$ are the points on a segment $\spol$ of $\rpol_\varphi$
of slope $-\frac{\slopenum}{\slopeden}$, then the residual polynomial of $\spol$ is
\[
\RA(x)
=
\sum_{i=k}^l\underline{\rho_{p^{s_i}}\alpha^{-J_i}}\,x^{(p^{s_i}-p^{s_k})/\slopeden}
=
\sum_{i=k}^l\underline{
\varphi_{b_i} {\tbinom{b_i}{p^{s_i}}}\alpha^{-a_i n-n}
}\,x^{(p^{s_i}-p^{s_k})/\slopeden}.
\]
\end{enumerate}
\end{proposition}

We immediately get:

\begin{corollary}\label{cor res pol seg}
Let $\varphi\in\OK[x]$ be Eisenstein and 
$\rpol_\varphi$ its ramification polygon.  
\begin{enumerate}
\item The residual polynomial of the rightmost segment of $\rpol_\varphi$ is monic.
\item
Let $(p^{s_l},J_l)$ 
be the right end point of the $i$-th segment of $\rpol_\varphi$ 
and $\RA_i=\sum_{j=0}^{m_i} \RA_{i,j}$ its residual polynomial
and let $(p^{s_k},J_k)$ 
be the left end point of the $(i+1)$-st segment of $\rpol_\varphi$
and $\RA_{i+1}=\sum_{j=0}^{m_{i+1}} \RA_{i+1,j}$ its residual polynomial.
Then $\RA_{i,m_i}=\RA_{i+1,0}$.
\end{enumerate}
\end{corollary}

We now give criteria for the existence of polynomials with given ramification polygon $\rpol$ and given residual polynomials.

\begin{proposition}\label{prop A}
Let $n=p^r e_0$ with $\gcd(p,e_0)=1$ and let
$\rpol$ be a polygon with points
\[
\rpol = \{ (1,J_0),(p^{s_1},J_1),\dots,(p^{s_k},J_k),\dots,(p^{r},0),\dots,(p^re_0,0) \}
\]
satisfying Proposition \ref{prop ram pol iff}.
Write $J_k=a_kn+b_k$ with $0\le b_k\le n$.
Let $\spol_1,\dots,\spol_\ell$ be the segments of $\rpol$ with endpoints
$(p^{k_i},J_{k_i})$ and $(p^{l_i},J_{l_i})$
and slopes $-\slopenum_i/\slopeden_i$ ($1\le i <\ell$).
For $1\le i <\ell$ let
$\RA_i(x)=\sum_{j=0}^{(p^{l_i}-p^{k_i})/{\slopeden_i}}\RA_{i,j}x^j\in\RK$.

There is an Eisenstein polynomial of degree $p^r e_0$ with ramification polygon $\rpol$
and segments $\spol_1,\dots,\spol_\ell$ with residual polynomials
$\RA_1,\dots,\RA_\ell\in\RK[x]$ if and only if
\begin{enumerate}
\item $\RA_{i,\deg\RA_i}=\RA_{i+1,0}$ for $1\le i< \ell$,
\item $\RA_{i,j}\ne0$ if and only if $j=(q-p^{s_{k_i}})/\slopeden_i$ for some 
$q\in\{p^{s_1},\dots,p^r\}$ with  $p^{k_i}\le q\le p^{l_i}$,
\item if for some $1\le t,q\le u$ we have $b_t=b_q$ and 
$s_{k_i}\le s_t\le s_{l_i}$
and
$s_{k_j}\le s_q\le s_{l_j}$
then
\[
\RA_{i,(p^{s_t}-p^{s_{k_i}})/\slopeden_i}
=
\underline{
{\tbinom{b_t}{p^{s_t}}} {\tbinom{b_t}{p^{s_q}}}^{-1}
(-\varphi_0)^{a_q-a_t}
}
\RA_{j,(p^{s_q}-p^{s_{k_j}})/\slopeden_j}.
\]
\end{enumerate}
\end{proposition}
\begin{proof}

Suppose that $\varphi$ is an Eisenstein polynomial of degree $p^re_0$
with ramification polygon $\rpol$ and segments $\spol_1,\dots,\spol_\ell$ 
with residual polynomials $\RA_1,\dots,\RA_\ell\in\RK[x]$.
Property (a) is given by Lemma \ref{lem res pol} \ref{lem res pol end} and 
property (b) is given by Proposition \ref{prop res pol seg} (b).
To establish property (c), suppose that for some $1\le t,q\le u$ we have $b_t=b_q$ and 
$s_{k_i}\le s_t\le s_{l_i}$
and
$s_{k_j}\le s_q\le s_{l_j}$.
From Proposition \ref{prop res pol seg}, we have that
\[
\RA_{i,(p^{s_t}-p^{s_{k_i}})/\slopeden_i}=\varphi_{b_{t}} {\tbinom{b_{t}}{p^{s_{t}}}}\alpha^{-a_{t} n-n}
\textnormal{ and }
\RA_{j,(p^{s_q}-p^{s_{k_j}})/\slopeden_j}=\varphi_{b_{q}} {\tbinom{b_{q}}{p^{s_{q}}}}\alpha^{-a_{q} n-n}.
\]
As $b_t=b_q$, we have that $\varphi_{b_t}=\varphi_{b_q}$. Since
\[
\RA_{i,(p^{s_t}-p^{s_{k_i}})/\slopeden_i} {\tbinom{b_t}{p^{s_t}}}^{-1} \alpha^{a_t n+n}
= \varphi_{b_{t}} = \varphi_{b_{q}} =
\RA_{j,(p^{s_q}-p^{s_{k_j}})/\slopeden_j} {\tbinom{b_t}{p^{s_q}}}^{-1} \alpha^{a_q n+n},
\]
we have
\[
\RA_{i,(p^{s_t}-p^{s_{k_i}})/\slopeden_i}
=
{\tbinom{b_t}{p^{s_t}}} {\tbinom{b_t}{p^{s_q}}}^{-1}
(-\varphi_0)^{a_q-a_t}
\RA_{j,(p^{s_q}-p^{s_{k_j}})/\slopeden_j}.
\]

Conversely, suppose that $\rpol$ is a ramification polygon with segments $\spol_1,\dots,\spol_\ell$
with residual polynomials $\RA_1,\dots,\RA_\ell\in\RK[x]$ with properties (a), (b), and (c) of the proposition.
Let $\psi$ be a polynomial in $\OK[x]$ with $\psi_{e_0 p^r}=1$, $v_\pi(\psi_0)=1$ and
\[
\underline\psi_{b_t,1+a_t-v_\pi\binom{b_t}{p^{s_t}}}=
\RA_{i,(p^{s_t}-p^{s_{k_i}})/\slopeden_i}
\underline{
\binom{b_t}{p^{s_t}}^{-1}(-\psi_{0,1})^{a_t+1}\pi^{v_\pi\binom{b_t}{p^{s_t}}}
} \textnormal{ for $i$ with }
p^{k_i}\le p^{s_t}\le p^{l_i}
\]
for each point $(p^{s_t},a_tn+b_t)$ in $\rpol$.
For $\psi$ to be well defined, we must check that the same coefficient is not assigned different values.
Multiple assignments occur at vertices (when one point contributes to two $\RA_i$)
and when multiple points have the same $b_t$.
If $(p^{s_t},a_tn+b_t)$ is a vertex of $\rpol$, then we have
\begin{align*}
\underline\psi_{b_t,1+a_t-v_\pi\binom{b_t}{p^{s_t}}}&=
\RA_{i,(p^{s_t}-p^{s_{k_i}})/\slopeden_i}
\underline{
\binom{b_t}{p^{s_t}}^{-1}(-\psi_{0,1})^{a_t+1}\pi^{v_\pi\binom{b_t}{p^{s_t}}}
}\\
&=
\RA_{i+1,(p^{s_t}-p^{s_{k_{i+1}}})/\slopeden_{i+1}}
\underline{
\binom{b_t}{p^{s_t}}^{-1}(-\psi_{0,1})^{a_t+1}\pi^{v_\pi\binom{b_t}{p^{s_t}}}
}.
\end{align*}
Cancellation gives us $\RA_{i,(p^{s_t}-p^{s_{k_i}})/\slopeden_i} = \RA_{i+1,(p^{s_t}-p^{s_{k_{i+1}}})/\slopeden_{i+1}}$.
As a vertex, $p^{s_t}$ is the abscissa of both the right endpoint of $\spol_i$ ($p^{s_{l_i}} = p^{s_t}$)
and the left endpoint of $\spol_{i+1}$ ($p^{s_{k_{i+1}}} = p^{s_t}$).
Thus $(p^{s_t}-p^{s_{k_i}})/\slopeden_i = \deg\RA_i$ and $(p^{s_t}-p^{s_{k_{i+1}}})/\slopeden_{i+1} = 0$.
So, $\RA_{i,\deg\RA_i} = \RA_{i+1,0}$, which is property (a).
On the other hand, if for some $1\le t,q\le u$, we have $b_t=b_q$,
with $s_{k_i}\le s_t\le s_{l_i}$
and
$s_{k_j}\le s_q\le s_{l_j}$,
then let $b=b_t=b_q$ and we have
\begin{align*}
\underline\psi_{b,1+a_t-v_\pi\binom{b_t}{p^{s_t}}}&=
\RA_{i,(p^{s_t}-p^{s_{k_i}})/\slopeden_i}
\underline{
\binom{b}{p^{s_t}}^{-1}(-\psi_{0,1})^{a_t+1}\pi^{v_\pi\binom{b}{p^{s_t}}}
}\\
\underline\psi_{b,1+a_q-v_\pi\binom{b}{p^{s_q}}}&=
\RA_{j,(p^{s_q}-p^{s_{k_j}})/\slopeden_j}
\underline{
\binom{b}{p^{s_q}}^{-1}(-\psi_{0,1})^{a_q+1}\pi^{v_\pi\binom{b}{p^{s_q}}}
}.
\end{align*}
As $\rpol$ is a ramification polygon, by Proposition \ref{prop ram pol iff} (b),
$b_t = b_q$ implies that $a_t = a_q - v_\pi\binom{b}{p^{s_q}} + v_\pi\binom{b}{p^{s_t}}$,
so we have that $1+a_t-v_\pi\binom{b}{p^{s_t}}=1+a_q-v_\pi\binom{b}{p^{s_q}}$.
These two assignments of coefficients of $\psi_b$ set the same coefficient,
and by property (c), they have the same value.
Thus, $\psi$ is well-defined, and we have set at most one $\pi$-adic coefficient for each
polynomial coefficient.

By property (b), none of the assigned coefficients are zero and no others are non-zero.
Thus, $v_\pi(\psi_{b_t})=1+a_t-v_\pi\binom{b_t}{p^{s_t}}$, and 
as per the construction in the proof of Proposition \ref{prop ram pol iff},
$\psi$ is an Eisenstein polynomial with ramification polygon $\rpol$.

Next we consider the residual polynomials of the segments of $\rpol$ as given by $\psi$.
Let $\spol_i$ be a segment of $\rpol$ containing points $(p^{s_k},J_k),\ldots,(p^{s_l},J_l)$ of slope $-h_i/e_i$.
Let $\RA^*_i$ be the residual polynomial of $\spol_i$.
From Proposition \ref{prop res pol seg}, for each point $(p^{s_t},a_t n+b_t)$ with $s_k\leq s_t\leq s_l$, we get
\[
\RA^*_{i,(p^{s_t}-p^{s_k})/e} = \underline{\psi_{b_t} \binom{b_t}{p^{s_t}} \alpha^{-a_tn-n}}.
\]
We need the right side to reduce to our intended value.
By our assignment,
\[\psi_{b_t}
= \RA_{i,(p^{s_t}-p^{s_{k_i}})/\slopeden_i}
\underline{
\binom{b_t}{p^{s_t}}^{-1}(-\psi_{0,1})^{a_t+1}\pi^{v_\pi\binom{b_t}{p^{s_t}}}
\pi^{1+a_t-v_\pi\binom{b_t}{p^{s_t}}}.
}
\]
With $\alpha^n\sim -\norm_{\KK(\alpha)/\KK}(\alpha)=-\psi_0\sim-\psi_{0,1}\pi$ we get
\[
\underline{\psi_{b_t} \tbinom{b_t}{p^{s_t}} \alpha^{-a_tn-n}}
=
\RA_{i,(p^{s_t}-p^{s_{k_i}})/\slopeden_i}
\underline{
\tbinom{b_t}{p^{s_t}}^{-1}(-\psi_{0,1})^{a_t+1}\pi^{v_\pi\binom{b_t}{p^{s_t}}}
\pi^{1+a_t-v_\pi\binom{b_t}{p^{s_t}}}
\tbinom{b_t}{p^{s_t}} (-\psi_{0,1}\pi)^{-a_t-1}
}
\]
from which cancellation gives us our desired result $\RA^*_{i,(p^{s_t}-p^{s_k})/e} = \RA_{i,(p^{s_t}-p^{s_k})/e}$.
\end{proof}

\subsection*{The invariant $\invA$ of $\KL/\KK$}

We introduce an invariant of $\KL/\KK$, 
that is compiled from the residual polynomials of the 
segments of the ramification polygon of $\varphi$.   
From the proof of \cite[Proposition 4.4]{greve-pauli} we obtain:

\begin{lemma}\label{lem res pol zeros}
Let $\varphi\in\OK[x]$ be Eisenstein and $\alpha$ a root of $\varphi$ and $\KL=\KK(\alpha)$.
Let $\spol$ be a segment of the ramification polygon of $\varphi$ of slope $-\slopenum/\slopeden$
and let $\RA$ be its residual polynomial.
Let $\beta=\delta\alpha$ with $v_\alpha(\delta)=0$ be another uniformizer of $\KL$ 
and $\psi$ its minimal polynomial.
If $\underline\gamma_1,\dots,\underline\gamma_m$ are the (not necessarily distinct) zeros of $\RA$ then 
$\underline\gamma_1/\underline\delta^h,\dots,\underline\gamma_m/\underline\delta^h$ are the
zeros of the residual polynomial of the segment of slope $-\slopenum/\slopeden$
of the ramification polygon of $\psi$.
\end{lemma}

Thus the zeros of the residual polynomials of all  segments of the ramification polygon change by powers of the same 
element $\underline\delta$ when transitioning from a uniformizer $\alpha$ to a uniformizer $\delta\alpha$.
With Proposition \ref{prop A} we obtain:
{
\begin{theorem}\label{theo A}
Let $\spol_1,\dots,\spol_\ell$ be the segments of the ramification polygon 
$\rpol$ of an Eisenstein polynomial  $\varphi\in\OK[x]$.
For $1\le i\le \ell$ let $-h_i/e_i$ be the slope of $\spol_i$
and $\RA_i(x)=\sum_{j=0}^{m_i}$ its residual polynomial.
Then 
\begin{equation}\label{eq A}
\invA=\left\{ 
\left(\gamma_{\delta,1}{\RA_1}(\underline\delta^{h_1} x),\dots,
\gamma_{\delta,\ell}{\RA_\ell}(\underline\delta^{h_\ell} x)\right) 
: 
\underline\delta\in\RK^\times
\right\}
\end{equation}
where 
$
\gamma_{\delta,\ell}=\delta^{-h_\ell\deg\RA_\ell},
$
and
$
\gamma_{\delta,i}=\gamma_{\delta,i+1}\delta^{-h_i\deg\RA_i}
$
for $1\le i\le \ell-1$
is an invariant of the extension $\KK[x]/(\varphi)$. 
\end{theorem}
}

\begin{example}\label{ex theo A}
Let $\varphi(x)=x^9+6x^3+9x+3$.  
The ramification polygon of $\phi$ consists of the two segments 
with end points $(1,10),(3,3)$ and $(3,3),(9,0)$
and residual polynomials $1+2x$ and $2+x^3$.
We get
\[
\invA=\{(1+2x,2+x^3),(1+x,1+x^3)\}.
\]
\end{example}

\subsection*{Generating Polynomials}
We show how the choice of a representative of 
the invariant $\invA$ determines some of the coefficients of the generating polynomials with this invariant.

\begin{lemma}\label{lem A fix}
Let $\varphi\in\OK[x]$ be Eisenstein of degree $n$. 
Let $\spol$ be a segment of 
ramification polygon
of $\varphi$ with endpoints
$(p^{s_k},a_k n+b_k)$ and
$(p^{s_l},a_l n+b_l)$ and residual polynomial $\RA(x)=\sum_{j=1}^{p^{s_l}-p^{s_k}}\RA_j x^j\in\RK[x]$.
If $(p^{s_i},a_i n+b_i)$
is a point on $\spol$ with $b_i \neq 0$ then
\[
\underline\varphi_{b_i,j}=
\RA_{(p^{s_i}-p^{s_k})/\slopeden}
\underline{
{\tbinom{b_i}{p^{s_i}}}^{-1}(-\varphi_{0,1})^{a_i+1}\pi^{v_\pi\binom{b_i}{p^{s_i}}}
}
\]
where $j=a_i+1-v_\pi\binom{b_i}{p^{s_i}}$.
\end{lemma}

\begin{proof}
By Lemma \ref{lem vphi}, $v_\pi(\varphi_{b_i}) = j$ and
by Proposition \ref{prop res pol seg}
\[
\RA(x)=\sum_{i=k}^l\underline{
\varphi_{b_i} {\tbinom{b_i}{p^{s_i}}}\alpha^{-a_i n-n}
}\,x^{(p^{s_i}-p^{s_k})/\slopeden}.
\]
Thus
$
\RA_{(p^{s_i}-p^{s_k})/\slopeden}=\varphi_{b_i} {\binom{b_i}{p^{s_i}}}\alpha^{-a_i n-n}
$.
With $\alpha^n\sim -\norm_{\KK(\alpha)/\KK}(\alpha)=-\varphi_0\sim-\varphi_{0,1}\pi$ we get
\[
\RA_{(p^{s_i}-p^{s_k})/\slopeden}=\varphi_{b_i} {\tbinom{b_i}{p^{s_i}}}(-\varphi_0)^{-a_i-1}.
\]
As by Lemma \ref{lem rho sim} 
$v_\alpha(\varphi_{b_i})=v_\alpha(\rho_{p^{s_i}})-v_\alpha\binom{b_i}{p^{s_i}}-b_i+n
=a_i n+b_i-v_\alpha\binom{b_i}{p^{s_i}}-b_i+n = n(a_i+1)-v_\alpha\binom{b_i}{p^{s_i}}$
we have 
$
\varphi_{b_i}\sim\varphi_{b_i,j}\pi^{a_i+1-v_\pi\binom{b_i}{p^{s_i}}}.
$
Therefore
\[
\RA_{(p^{s_i}-p^{s_k})/\slopeden}=
\underline{\varphi_{b_i,j}{\tbinom{b_i}{p^{s_i}}}(-\varphi_{0,1}\pi)^{-a_i-1}\pi^{a_i+1-v_\pi\binom{b_i}{p^{s_i}}}}
=\underline{
\varphi}_{b_i,j}
(-\underline{\varphi}_{0,1})^{-a_i-1}
\underline{
{\tbinom{b_i}{p^{s_i}}}
\pi^{-v_\pi\binom{b_i}{p^{s_i}}}.
}\qedhere
\]
\end{proof}

A change of the uniformizer $\alpha$ of $\KL=\KK(\alpha)$ to $\delta\alpha$ with $v(\delta)=0$ that determines the representative
$(\RA_1,\dots,\RA_\ell)\in\invA$ also effects the constant coefficient of the generating polynomial.
Namely if the Eisenstein polynomial $\varphi=x^n+\sum_{i=0}^{n-1}\varphi_i x^i\in\OK[x]$ is the minimal polynomial of $\alpha$ then
$
\psi(x)=\delta^n \varphi\left(\frac{x}{\delta}\right)
$
with $\psi_{0,1}=\delta^n\varphi_{0,1}$
is the minimal polynomial of $\delta\alpha$.

\begin{lemma}\label{lem phi0}
Let $\varphi\in\OK[x]$ be Eisenstein of degree $n$ and $\RS{0}: \RK\to\RK, a\mapsto a^n$.
\begin{enumerate}
\item If and only if $\underline\delta\in\RS{0}(\RK)$, there is $\psi\in\OK[x]$ Eisenstein with $\underline{\psi}_{0,1}=\underline{\delta}\underline{\varphi}_{0,1}$ 
such that $\KK[x]/(\psi)\cong\KK[x]/(\varphi)$.
\item If $n=p^r$ for some $r\in\N$ then $\RS{0}$ is surjective and there is
$\psi\in\OK[x]$ Eisenstein with $\underline{\psi}_{0,1}=1$
such that $\KK[x]/(\psi)\cong\KK[x]/(\varphi)$.
\end{enumerate}
\end{lemma}

This corresponds to the reduction step 0 in Monge's reduction \cite[Algorithm 1]{monge}.
If $n=p^r e_0$ with $\gcd(p,e_0)=1$ then $\underline\varphi_{0,1}$ determines the tamely ramified subextensions
of $\KK[x]/(\varphi)$, that can be generated by $x^{e_0}+\varphi_{0,1}\pi$.

If we fix $\varphi_{0,1}$ then the set of representatives of
$\invA$ becomes
\begin{equation}\label{eq invAstar}
\invA^*=
\left\{
\left(\gamma_{\delta,1}{\RA_1}(\underline\delta^{h_1} x),\dots,
\gamma_{\delta,\ell}{\RA_\ell}(\underline\delta^{h_\ell} x)\right) 
: 
\underline\delta\in\RK^\times, \underline\delta^n=1
\right\}
\end{equation}
where 
$
\gamma_{\delta,\ell}=\delta^{-h_\ell\deg\RA_\ell},
$
and
$
\gamma_{\delta,i}=\gamma_{\delta,i+1}\delta^{-h_i\deg\RA_i}
$
for $1\le i\le \ell-1$.
Thus fixing $\varphi_{0,1}$ yields a partition of $\invA$.
Also, if $n$ is a power of $p$ then $\invA^*$ contains exactly one representative of $\invA$.

\begin{remark}\label{rem A star}
Let a ramification polygon $\rpol$ and $\RA_{1},\dots,\RA_\ell\in\RK[x]$
satisfying Proposition \ref{prop A}.  Let $\invA$ as in Theorem \ref{theo A}
and
$\invA=\invA^{*1}\cup\dots\cup\invA^{*k}$ be the partition of $\invA$ 
into sets as in Equation (\ref{eq invAstar}). 
Let $\underline\gamma\in\RK^\times$.
Then there is no transformation $\delta\alpha$ of the uniformizer $\alpha$ of an extension with 
$\rpol$ and residual polynomials in $\invA^{*i}$ for some $1\le i\le k$ generated by $\varphi\in\OK[x]$ 
with $\underline\varphi_{0,1}=\underline\gamma$ such that the residual polynomials of the segments of
$\rpol_\varphi=\rpol$ is not in $\invA^{*i}$.
Thus the construction of generating polynomials for all extensions with $\rpol$ and $\invA$
can be reduced to constructing polynomials with residual polynomials in the sets $\invA^{*i}$.
\end{remark}

\begin{lemma}\label{lem A star}
Let $(\RA_{1},\dots,\RA_\ell)\in\invA^*$.
If $\psi\in\OK[x]$ is a polynomial with residual polynomials in $\invA^*$,
then there is a polynomial $\varphi\in\OK[x]$ with residual polynomials $(\RA_{1},\dots,\RA_\ell)$ such that
$\KK[x]/(\psi)\cong\KK[x]/(\varphi)$.
\end{lemma}
\begin{proof}
Let $\RA'_{1},\dots,\RA'_\ell$ be the residual polynomials of $\psi$.
As $(\RA'_{1},\dots,\RA'_\ell)\in\invA^*$ there exists a $\underline\delta\in\RK^\times$ with $\underline\delta^n=1$ so that
\[
(\RA_{1},\dots,\RA_\ell) =
\left(\gamma_{\delta,1}{\RA'_1}(\underline\delta^{h_1} x),\dots,
\gamma_{\delta,\ell}{\RA'_\ell}(\underline\delta^{h_\ell} x)\right).
\]
where $ \gamma_{\delta,\ell}=\delta^{-h_\ell\deg\RA_\ell}, $
and $ \gamma_{\delta,i}=\gamma_{\delta,i+1}\delta^{-h_i\deg\RA_i} $ for $1\le i\le \ell-1$.

Let $\alpha$ be a root of $\psi$ and
$ \varphi(x)=\delta^n \psi(\delta^{-1} x) $
be the minimal polynomial of $\delta\alpha$.
This gives us that $\KK[x]/(\psi)\cong\KK[x]/(\varphi)$.

Let us find the residual polynomials of $\varphi$.
From Proposition \ref{prop res pol seg}, we have that the residual polynomial for a segment $\spol_i$
of slope $h/e$ with endpoints $(p^{s_{k_i}},J_{k_i}=a_{k_i} n+b_{k_i})$ and $(p^{s_{l_i}},J_{l_i}=a_{l_i} n+b_{l_i})$ is
\[
\sum_{j=k_i}^{l_i}\underline{
\varphi_{b_j} {\tbinom{b_j}{p^{s_j}}}\alpha^{-a_j n-n}
}\,x^{(p^{s_j}-p^{s_{k_i}})/\slopeden}.
\]
Performing our substitution we have that this polynomial is
\[
\sum_{j=k_i}^{l_i}\underline{
\delta^{n-b_j}\psi_{b_j} {\tbinom{b_j}{p^{s_j}}}(\delta\alpha)^{-a_j n-n}
}\,x^{(p^{s_j}-p^{s_{k_i}})/\slopeden}
=
\sum_{j=k_i}^{l_i}\underline{
\delta^{n-b_j-a_jn-n}
} \; \RA'_{i,j}
=
\sum_{j=k_i}^{l_i}\underline{
\delta^{-J_j}
} \; \RA'_{i,j}.
\]

Next, let us perform the deformation of $\RA'_i$ by $\delta$.
First, we consider $\gamma_{\delta,i}$.
Notice that for the $\RA'_i$,
the residual polynomial of the segment $\spol_i$ with endpoints $(p^{s_{k_i}},J_k)$ and $(p^{s_{l_i}},J_l)$,
\[
\underline\delta^{-h_i\deg\RA'_i}
= \underline\delta^{\lambda_i(p^{s_{l_i}}-p^{s_{k_i}})}
= \underline\delta^{J_{l_i}-J_{k_i}}
= \left\lbrace \begin{array}{ll}
\underline\delta^{J_{l_1}-J_{k_1}} &\textrm{ if } i = 1 \\
\underline\delta^{J_{l_i}-J_{l_{i-1}}} &\textrm{ if } 2\leq i < \ell \\
\underline\delta^{-J_{l_{\ell-1}}} = \underline\delta^{-J_{k_{\ell}}} &\textrm{ if } i = \ell
\end{array}\right..
\]
This shows us that for $1\le i\le \ell-1$, $\gamma_{\delta,i}=\gamma_{\delta,i+1}\delta^{-h_i\deg\RA'_i} = \underline\delta^{-J_{k_i}}$,
and in general, $\gamma_{\delta,i}=\underline\delta^{-J_{k_i}}$.
So the deformation of $\RA'_i$ by $\delta$ is
\[
\RA_i = 
\gamma_{\delta,i} \RA'_{i,j}(\delta^{h_i}x) = 
\underline\delta^{-J_{k_i}} \sum_{j=k_i}^{l_i} \RA'_{i,j} \underline{\delta^{-\lambda_i (p^{s_j}-p^{s_{k_i}})}} =
\underline\delta^{-J_{k_i}} \sum_{j=k_i}^{l_i} \RA'_{i,j} \underline{\delta^{-J_j + J_{k_i}}} =
\sum_{j=k_i}^{l_i}\underline{ \delta^{-J_j} } \; \RA'_{i,j}.
\]
Thus, the residual polynomials of $\varphi(x)$ are $(\RA_{1},\dots,\RA_\ell)$
and $\KK[x]/(\psi)\cong\KK[x]/(\varphi)$.
\end{proof}

\begin{example}[Example \ref{ex R2} continued]\label{ex R2 A21}
Let $\rpol_2=\{(1,10),(3,3),(9,0)\}$.
There are two choices for the invariant $\invA$, namely
$\invA_{2,1}=\{(1+2x, 2+x^3),(1+x,1+x^3)\}$ (compare Example \ref{ex theo A})
and
$\invA_{2,2}=\{(2+2x, 2+x^3),(2+x,1+x^3)\}$.

By Lemma \ref{lem phi0} all extensions of $\Q_3$ with ramification polygon $\rpol$ can
be generated by polynomials $\varphi\in\Z_3[x]$ with \nbt{$\varphi_0\equiv 3\bmod 9$}.
Fixing $\varphi_{0,1}=1$ gives the partition 
$\invA_{2,1}=\invA^{*1}_{2,1}\cup\invA^{*2}_{2,1}$ with
$\invA^{*1}_{2,1}=\{(1+2x, 2+x^3)\}$
and $\invA^{*2}_{2,1}=\{(1+x,1+x^3)\}$.

For the generating polynomials of the fields with 
$\invA^{*1}_{2,1}$ by Lemma \ref{lem A fix}
we get, 
from the point $(1,10)=(3^0,1\cdot9+1)$ on $\rpol_2$ that \ngt{$\varphi_{1,2}=1$} and 
from the point $(3,3)=(3^1,0\cdot9+3)$ on $\rpol_2$ that \nrt{$\varphi_{3,1}=2$}.
The polynomials given by $\rpol_2$ and $\invA^{*1}$ are described by:

\begin{center}
\small
\renewcommand\arraystretch{1.1}
\renewcommand\tabcolsep{2pt}
\begin{tabular}{l|cccccccccc}
             &$x^9$    &$x^8$      &$x^7$      &$x^6$      &$x^5$      &$x^4$      &$x^3$      &$x^2$      &$x^1$        &$x^0$ \\\hline
      $3^4$&$\{0\}$    &\ote{\{0\}}&\ote{\{0\}}&\ote{\{0\}}&\ote{\{0\}}&\ote{\{0\}}&\ote{\{0\}}&\ote{\{0\}}&\ote{\{0\}}  &\ote{\{0\}}  \\
      $3^3$&$\{0\}$    &$\{0,1,2\}$&$\{0,1,2\}$&$\{0,1,2\}$&$\{0,1,2\}$&$\{0,1,2\}$&$\{0,1,2\}$&$\{0,1,2\}$&$\{0,1,2\}$  &$\{0,1,2\}$  \\
      $3^2$&$\{0\}$    &$\{0,1,2\}$&$\{0,1,2\}$&$\{0,1,2\}$&$\{0,1,2\}$&$\{0,1,2\}$&$\{0,1,2\}$&$\{0,1,2\}$&\nge{\{1\}}&$\{0,1,2\}$  \\
      $3^1$&$\{0\}$    &\ote{\{0\}}&\ote{\{0\}}&$\{0,1,2\}$&\ote{\{0\}}&\ote{\{0\}}&\nre{\{2\}}&\ote{\{0\}}&\ote{\{0\}}&\nbe{\{1\}}  \\
      $3^0$&$\{1\}$    &$\{0\}$    &$\{0\}$    &$\{0\}$    &$\{0\}$    &$\{0\}$    &$\{0\}$    &$\{0\}$    &$\{0\}$      &$\{0\}$
\end{tabular}
\end{center}

By Remark \ref{rem A star}, 
proceeding as above with $\invA^{*2}_{2,1}$ yields a template for generating polynomials for the remaining extensions
with ramification polygon $\rpol$ and invariant $\invA$.
\end{example}


\section{Residual Polynomials of Components}\label{sec res comp}

We now apply some results of Monge \cite{monge} to reduce the number of polynomials that we
need to consider to generate all extensions with given invariants.

\begin{definition}
Let $\npol$ be a Newton polygon.
For $\lambda\in\Q$ we call
\[
\cpol{\lambda} = \bigl\{(k,w)\in\npol\mid \lambda k+w=\min\{\lambda l+u \mid (l,u)\in\npol\}\bigr\}
\]
the $\lambda$-component of $\npol$.
\end{definition}

\begin{remark}\label{rem comp}
If $\npol$ has a segment with slope $\lambda$ then $\cpol{\lambda}$ contains that segment. 
Otherwise $\cpol{\lambda}$ consists of only one point.  
\end{remark}

To each component of integral slope of a ramification polygon we attach a residual polynomial.

\begin{definition}\label{def Sm}
Let $\varphi\in\OK[x]$ be Eisenstein, $\alpha$ a root of $\varphi$, 
$\rho$ the ramification polynomial of $\varphi$,
and $\rpol$ the ramification polygon of $\varphi$.
For $\lambda\in\N$ the residual polynomial of the $(-\lambda)$-component of $\rpol$ is
\[
\RS{\lambda}(x)=\underline{\rho(\alpha^{\lambda} x)/\cont{\alpha}{\rho(\alpha^\lambda x)}}
\]
where
$\cont{\alpha}{\rho(\alpha^\lambda z)}$ denotes the highest power of $\alpha$ dividing all coefficients of $\rho(\alpha^\lambda z)$.
\end{definition}

The quantity $\cont{\alpha}{\rho(\alpha^m z)}$ only depends on the ramification polygon.
Namely if $\rho(x)=\sum_{i=1}^n\rho_i x^i$ we have
$\rho(\alpha^\lambda x) = \sum_{i=0}^n \rho_i (\alpha^\lambda x)^i
  = \sum_{i=0}^n \rho_i (\alpha^\lambda)^i x^i$ and obtain
\[
n\hhf_\rpol(\lambda)
=\min_{0\le i\le n} v(\rho_i)+i\lambda=
\cont{\alpha}{\rho(\alpha^\lambda x)}
\]
for the Hasse-Herbrand function 
$\hhf_\rpol$
of $\rpol$  (Definition \ref{def hhf}).
Thus  \cite[Proposition 1]{monge} yields 
\[
n\hhf_\rpol(\lambda)=\cont{\alpha}{\rho(\alpha^\lambda x)}=n\hhf_{L/K}(\lambda).
\]
{To calculate $n\hhf_\rpol(\lambda)$, we only have to take the minimum of the $v(\rho_i)+i\lambda$ for the points $(v(\rho_i),i)$ 
on the polygon.
For $p^s<i<p^{s+1}$, we have $v_\alpha(\rho_{p^s}) \leq v_\alpha(\rho_i)$ (Lemma \ref{lem scherk} (c))
and $p^s < i$, which gives us that $v_\alpha(\rho_{p^s})+p^s \lambda < v_\alpha(\rho_i)+i\lambda$.
This demonstrates the formula for $\hhf_\rpol$ from Definition \ref{def hhf}.}

\begin{lemma}\label{lem comp res poly}
Let $\rpol$ be the ramification polygon of $\varphi$.
\begin{enumerate}
\item
If $\rpol$ has a segment $\spol$ of integral slope $-m\in\Z$, 
with left endpoint $(k,w)$  and residual polynomial $\RA$
then $\RS{m}(x)=x^{k}\RA(x)$.
\item If $\rpol$ has no segment of slope $-m\in\Z$ then
$\RS{m}(x)=x^{p^s}$ where $0\le s \le v_p(n)$ such that 
$v(\rho_{p^s})+p^s\cdot m=\min_{0\le r \le v_p(n)}v(\rho_{p^r})+p^r\cdot m$.
\item\label{lem comp res poly additive} For all $m\in\N$ the residual polynomial 
$\RS{m}$ of $\rpol_{-m}$ is an additive polynomial.
\item $\RS{m}:\RK\to\RK$ is $\F_p$-linear.
\end{enumerate}
\end{lemma}

\begin{proof}
\begin{enumerate}
\item
By Remark \ref{rem comp}
the component $\rpol_{(-m)}$ contains $\spol$ and
by Remark \ref{lem res pol}(\ref{lem res pol int})
$\RS{m}(x)=x^{k}\RA(x)$.
\item
As mentioned in Remark \ref{rem comp}
$\cpol{(-m)}$ and $\rpol$ only have one point in common.  
By Lemma \ref{lem scherk} this point 
is of the form $(p^s,v(\rho_{p^s}))$.
It follows from 
Lemma \ref{lem scherk} 
that if
the ramification polygon $\rpol$ of $\varphi$ has no segment of slope $-m$  then
\[
v\left(\cont{\alpha}{\rho(\alpha^m x)}\right)
=\min_{0\le i\le n}
v(\rho_i)+i\cdot m
=\min_{0\le r \le v_p(n)}
v(\rho_{p^r})+p^r\cdot m
\]
and $\RS{m}(x)=x^{p^s}$ where $0\le s \le v_p(n)$ such that 
$v(\rho_{p^s})+p^s\cdot m=\min_{0\le r \le v_p(n)}v(\rho_{p^r})+p^r\cdot m$.
\item By Lemma \ref{lem scherk} the abscissa of each point on $\rpol$ is of the form 
$p^s$.  Thus the residual polynomial of $\rpol_{(-m)}$ is the sum of monomials of the form
$x^{p^s}$ which implies that $\RS{m}$ is additive.
\item Is a direct consequence of \ref{lem comp res poly additive}.
\end{enumerate}
\end{proof}

We now investigate the effect of changing the uniformizer $\alpha$ of $\KK(\alpha)$ on the coefficients
of its  minimal polynomial (compare {\cite[Lemma 3]{monge}}).

\begin{proposition}\label{prop other uni}
Let $\varphi\in\OK[x]$ be Eisenstein of degree $n$, let $\alpha$ be a root of $\varphi$ 
and let $\rho$ be the ramification polynomial of $\varphi$.
Let $\beta=\alpha+\gamma\alpha^{m+1}$ where $\gamma\in\KL=\KK(\alpha)$ with $v(\gamma)=0$ 
be another uniformizer of $\KL$ and $\psi\in\OK[x]$ its minimal polynomial.
\begin{enumerate}
\item If $0\le j<n$ and $j\equiv v_\alpha\left(\rho(\gamma\alpha^m)\right) \bmod n$ then 
$\varphi_j-\psi_j=\alpha^n\rho(\gamma\alpha^m)$
\item If $0\le k< n$ and $k\equiv v_\alpha(\cont{\alpha}{\rho(\alpha^m x)})\bmod n$ then
\[
\underline{(\varphi_k-\psi_k)/(\alpha^{n-k}\cont{\alpha}{\rho(\alpha^m x)})}
=\RS{m}(\underline\gamma).
\]
\end{enumerate}
\end{proposition}

\begin{proof}
\begin{enumerate}
\item By Definition \ref{def ram pol} we have
\begin{equation}\label{eq phibeta}
\sum_{i=0}^{n-1} (\varphi_i-\psi_i)\beta^i = 
\varphi(\beta) - \psi(\beta)= 
\varphi(\beta) = 
\alpha^n \rho(\beta/\alpha-1)=
\alpha^n \rho(\gamma\alpha^m).
\end{equation}
Since
$v_\pi(\varphi_i)\in\Z$ and $v_\pi(\psi_i)\in\Z$ and $v_\pi(\beta^i)=\frac{i}{n}$
we have
\[
v_\pi\left(\sum_{i=0}^{n-1}(\varphi_i-\psi_i)\beta^i\right)=
\min_{0\le i<n-1}v_\pi\left((\varphi_i-\psi_i)\beta^i\right).
\]
Thus for $0\le j<n$ and $j\equiv v_\pi\left(\rho(\gamma\alpha^m)\right) \bmod n$ we have
$\varphi_j-\psi_j=\alpha^n\rho(\gamma\alpha^m)$.
\item
Dividing Equation (\ref{eq phibeta}) by $\alpha^n\cont{\alpha}{\rho(\alpha^m x)}$ yields
\[
\underline{
\left(
{\varphi(\beta)-\psi(\beta)}
\right)
/
\left(
{\alpha^n\cont{\alpha}{\rho(\alpha^m x)}}
\right)
}=
\underline{
{\alpha^n \rho(\gamma\alpha^m)}
/
\left(
{\alpha^n\cont{\alpha}{\rho(\alpha^m x)}}
\right)
}=
\RS{m}(\underline\gamma).
\]
For $0\le k< n$ with
$k\equiv v(\cont{\alpha}{\rho(\alpha^m x)})\bmod n$ we get
\[
\underline{(\varphi_k-\psi_k)\beta^k/(\alpha^n\cont{\alpha}{\rho(\alpha^m x)})}
=\RS{m}(\underline\gamma).
\]
With $\beta\equiv\alpha\bmod(\alpha^2)$  we obtain the result. \qedhere
\end{enumerate}
\end{proof}

\subsection*{Generating Polynomials}
Using the results from above we can reduce the set of generating polynomials with given 
invariants considerably.  
We show how the coefficients of a generating polynomial can be changed by changing the
uniformizer.  The coefficients that we can change arbitrarily this way we set to 0, thus reducing the number of polynomials to be considered.

\begin{corollary}\label{cor Sm}
Let $\varphi\in\OK[x]$ be Eisenstein of degree $n$, let 
$\alpha$ be a root of $\varphi$, let $\KL=\KK(\alpha)$, and let
$\rho$ be the ramification polynomial of $\varphi$.
Let $m\in\N$, 
$c=v_\alpha(\cont{\alpha}{\rho(\alpha^m x)})$,
$0\le k< n$ with $k\equiv c\bmod n$, and $j=\frac{n-k+c}{n}$.
\begin{enumerate}
\item
If $\underline\delta\in\RS{m}(\RK)$ then 
for the minimal polynomial $\psi\in\OK[x]$
of 
$\beta=\alpha+\gamma\alpha^{m+1}$ where $\underline\gamma\in\RS{m}^{-1}(\{\underline\delta\})$
we have $\underline\psi_{k,j}=\underline\varphi_{k,j}-\underline\delta$.
\item
If $\RS{m}:\RK\to\RK$ is surjective we can set 
$\underline\delta=\underline\varphi_{k,j}$ and obtain
$\underline\psi_{k,j}=0$.
\item
If $\RS{m}(\underline\gamma)=0$ 
and $d=v_\alpha(\alpha^n{\rho(\gamma\alpha^m )})$,
$0\le l< n$ with $l\equiv d\bmod n$, and $i=\frac{n-l+d}{n}$
then
$\underline\psi_{l,i}=\underline\varphi_{l,i}-\underline{\pi^{-i}\alpha^n{\rho(\gamma\alpha^m)}}$.
\end{enumerate}
\end{corollary}

The next Lemma follows directly from Corollary \ref{cor Sm}.

\begin{lemma}\label{lem Sm choice}
Let $\varphi\in\OK[x]$ be Eisenstein of degree $n$, $\rpol$ its ramification polygon.
Assume there is $m\in\N$ such that $k\equiv n\hhf_\rpol(m)\bmod n$ and $j=\frac{n+n\hhf_\rpol(m)-k}{n}$
and let $\RS{m}$ be the residual polynomials of $\rpol_{(-m)}$.
\begin{enumerate}
\item If $\RS{m}$ is surjective then there is an Eisenstein  polynomial $\psi\in\OK[x]$ with $\psi_{k,j}=0$.
such that $\KK[x]/(\psi)\cong\KK(\alpha)$.
\item If $\psi\in\OK[x]$ has the same ramification polygon with the same residual polynomials as $\varphi$ and
$\varphi_{k,j}-\psi_{k,j}\notin\RS{m}(\RK)$ then $\KK[x]/(\psi)\not\cong\KK[x]/(\varphi)$.
\end{enumerate}
\end{lemma}

\begin{example}[Example \ref{ex R2 A21} continued]\label{ex R2 A21 S}
The ramification polygon $\rpol_2=\{(1,10),(3,3),(9,0)\}$ has no segments with integral slope.
We get $\RS{1}=x^3$, $\RS{2}=x^3$, and $\RS{3}=x^3$, 
with $9\hhf(1)=6$, $9\hhf(2)=9$, and $9\hhf(3)=12$.
Thus \nrt{$\varphi_{6,1}=0$}, \nrt{$\varphi_{0,2}=0$}, and \nrt{$\varphi_{3,2}=0$}.
Furthermore $\RS{m}=x$ for with $9\hhf(m)=10+m$ for $m\ge 4$.
Thus by Lemma \ref{lem Sm choice} we can set \nbt{$\varphi_{k,j}=0$} 
for $k+9(j-1)\ge 14$.

For the generating polynomials with $\invA^{*1}_{2,1}$ we get the template:
\begin{center}
\small
\begin{tabular}{l|cccccccccc}
             &$x^9$    &$x^8$      &$x^7$      &$x^6$      &$x^5$      &$x^4$      &$x^3$      &$x^2$      &$x^1$        &$x^0$ \\\hline
      $3^4$&$\{0\}$    &\ote{\{0\}}&\ote{\{0\}}&\ote{\{0\}}&\ote{\{0\}}&\ote{\{0\}}&\ote{\{0\}}&\ote{\{0\}}&\ote{\{0\}}  &\ote{\{0\}}  \\
      $3^3$&$\{0\}$    &\nbe{\{0\}}&\nbe{\{0\}}&\nbe{\{0\}}&\nbe{\{0\}}&\nbe{\{0\}}&\nbe{\{0\}}&\nbe{\{0\}}&\nbe{\{0\}}  &\nbe{\{0\}}  \\
      $3^2$&$\{0\}$    &\nbe{\{0\}}&\nbe{\{0\}}&\nbe{\{0\}}&\nbe{\{0\}}&$\{0,1,2\}$&\nre{\{0\}}&\ote{\{0,1,2\}}&\ote{\{1\}}&\nre{\{0\}}  \\
      $3^1$&$\{0\}$    &\ote{\{0\}}&\ote{\{0\}}&\nre{\{0\}}&\ote{\{0\}}&\ote{\{0\}}&\ote{\{2\}}&\ote{\{0\}}&\ote{\{0\}}&\ote{\{1\}}  \\
      $3^0$&$\{1\}$    &$\{0\}$    &$\{0\}$    &$\{0\}$    &$\{0\}$    &$\{0\}$    &$\{0\}$    &$\{0\}$    &$\{0\}$      &$\{0\}$
\end{tabular}
\end{center}
Since changing the uniformizer cannot change $\varphi_{2,2}$ and $\varphi_{4,2}$ independently from
the other coefficients of $\varphi$
we obtain a unique generating polynomial of each extension with ramification polygon $\rpol_2$ and 
$\invA^{*1}_{2,1}$.
\end{example}


\section{Enumerating Generating Polynomials}\label{sec gen}

We use the results from the previous sections to formulate an algorithm that returns generating polynomials
of all extensions with given ramification polynomials and residual polynomials.
In certain cases this set will contain exactly one polynomial for each extension.

\Algo{AllExtensionsSub} 
{alg all sub} 
{A $\pi$-adic field $\KK$, a convex polygon $\rpol$ with points
$(1,a_0n+b_0)$, $(p^{s_1},a_1 n+b_1)$,\dots,$(p^{s_{u}},a_{u}n+b_{u})=(p^{s_u},0)$,\dots,$(n,0)$ 
satisfying Proposition \ref{prop ram pol iff}
where $0\le b_i <n$ for $1\le i\le u=v_p(n)$,
$\spol_1,\dots,\spol_\ell$ the segments of $\rpol$,
a representative $\underline\delta_0$ of a class in $\RK^\times/(\RK^\times)^n$,
and $\RA_1,\dots,\RA_\ell\in\RK[x]$ satisfying Proposition \ref{prop A}.
} 
{A set that contains at least one Eisenstein polynomial for each totally ramified extension of degree $n$,
that can be generated by a polynomial $\varphi$ with ramification polygon $\rpol$, 
$\underline\varphi_{0,1}=\underline\delta_0$, and residual polynomials $\RA_1,\dots,\RA_\ell$.
} 
{
\begin{enumerate}
\item $c \gets \left\lceil 1+2a_0+\frac{2b_0}{n} \right\rceil - 1$\hfill[Lemma \ref{lem krasner bound}]
\item Initialize template $(\tau_{i,j})_{0\le i\le n-1,1\le j\le c}$ with $\tau_{i,j}=\{\underline0\}\subset\RK$
\item For $0\le i\le n-1$ and $L_\rpol(i)\le j\le c$:\hfill[Definition \ref{def ells}]
\begin{itemize}
\item If there is no $m\in\N$ with $i\equiv n\hhf_\rpol(m)\bmod n$ and $j=\frac{n-i+n\hhf_\rpol(m)}{n}$:
\begin{itemize}
\item $\tau_{i,j}\gets\RK$.
\end{itemize}
\end{itemize}
\item For $1\le m\le \left\lfloor\frac{(a_1n+b1)-(a_0 n+b_0)}{p^{s_1}-1}\right\rfloor$:
\begin{itemize}
\item $i\gets n\hhf_\rpol(m)\csmod n$, $j\gets \frac{n-i+n\hhf_\rpol(m)}{n}$
\item $\tau_{i,j}\gets R$ where $R$ is a set of representatives of $\RK/\RS{m}(\RK)$.\hfill[Lemma \ref{lem Sm choice}]
\end{itemize}
\item For $1\le i \le u$:
\begin{itemize}
\item Find a segment $\spol_t$ of $\rpol$ such that 
$(p^{s_i},a_i n +b_i)$ is on $\spol_t$.
\item {$j\gets a_i+1-v_\pi\binom{b_i}{p^{s_i}}$}
\item 
{$\tau_{b_i,j}\gets
\left\{
\RA_{t,(p^{s_i}-p^{s_k})/\slopeden}
(-\underline\delta_{0})^{a_i+1}
\underline{\tbinom{b_i}{p^{s_i}}^{-1}
\pi^{v_\pi\binom{b_i}{p^{s_i}}}}\right\}.$}\hfill[Lemma \ref{lem A fix}]\\
where $(p^{s_k},a_k n +b_k)$ is the left end point of $\spol_t$ and $-\slopenum/\slopeden$ is the slope of $\spol_t$.
\end{itemize}
\item $\tau_{0,1}\gets \{\underline\delta_0\}$ \hfill[Lemma \ref{lem phi0}]
\item Return {$\left\{x^n+\sum_{i=0}^{n-1} \left(\sum_{j=1}^{c}\varphi_{i,j}\pi^j\right)x^i \in\OK[x]: 
\mbox{ $\varphi_{i,j}\in\reps{\RK}$ such that
$\underline\varphi_{i,j}\in\tau_{i,j}$}
\right\}$}
\end{enumerate}
}

As is evident from the following example Algorithm \ref{alg all sub} may return more than one generating polynomial for some extensions.
\begin{example}\label{ex R3}
The polygon $\rpol_3=\{(1,10),(3,6),(9,0)\}$ has segments with slopes $\frac{10-6}{1-3}=-2$
and $\frac{6-0}{3-9}=-1$.
With the choice $\varphi_0\equiv 3\bmod 9$
the possible pairs of residual polynomials are 
$\invA_{3,1}=\{(2+x^2, 1+x^6)\}$,
$\invA_{3,2}=\{(2+2x^2, 2+x^6)\}$,
$\invA_{3,3}=\{(1+2x^2, 2+x^6)\}$, and
$\invA_{3,4}=\{(1+x^2, 1+x^6)\}$.

For $\invA_{3,2}=\{(2+2x^2, 2+x^6)\}$ we get
$\varphi_{1,2}=2$ and 
furthermore this choice also gives
$\RS{1}=(2+x^6)x^3$, $\RS{2}=(2x^2+2)x=2(x^3+x)$, and $\RS{m}=x$ for $m\ge 3$
with $\RS{1}(\F_3)=\{0\}$, $\RS{2}(\F_3)=\F_3$, and $\RS{m}(\F_3)=\F_3$.
As $\RS{2}$ is surjective we can set \nbt{$\varphi_{3,2}=0$}.
As $\RS{m}$ for $m\ge 3$ we can set \nbt{$\varphi_{k,j}=0$} for $k+9(j-1)\ge 14$ where $0\le k<9$.
As the image of $\RS{1}$ is $\{\underline0\}$ changing the uniformizer does not affect \nrt{$\varphi_{0,2}$}.
Thus Algorithm \ref{alg all sub} generates the template:
\begin{center}
\small
\begin{tabular}{l|cccccccccc}
             &$x^9$    &$x^8$      &$x^7$      &$x^6$      &$x^5$      &$x^4$      &$x^3$      &$x^2$      &$x^1$        &$x^0$ \\\hline
      $3^4$&$\{0\}$    &\nbe{\{0\}}&\nbe{\{0\}}&\nbe{\{0\}}&\nbe{\{0\}}&\nbe{\{0\}}&\nbe{\{0\}}&\nbe{\{0\}}&\nbe{\{0\}}  &\nbe{\{0\}}  \\
      $3^3$&$\{0\}$    &\nbe{\{0\}}&\nbe{\{0\}}&\nbe{\{0\}}&\nbe{\{0\}}&\nbe{\{0\}}&\nbe{\{0\}}&\nbe{\{0\}}&\nbe{\{0\}}  &\nbe{\{0\}}  \\
      $3^2$&$\{0\}$    &\nbe{\{0\}}&\nbe{\{0\}}&\nbe{\{0\}}&\nbe{\{0\}}&\ote{\{0\}}&\nbe{\{0\}}&\ote{\{0,1,2\}}&\nge{\{2\}}&\nre{\{0,1,2\}}  \\
      $3^1$&$\{0\}$    &\ote{\{0\}}&\ote{\{0\}}&\nge{\{2\}}&\ote{\{0\}}&\ote{\{0\}}&\ote{\{0\}}&\ote{\{0\}}&\ote{\{0\}}&\ote{\{1\}}  \\
      $3^0$&$\{1\}$    &$\{0\}$    &$\{0\}$    &$\{0\}$    &$\{0\}$    &$\{0\}$    &$\{0\}$    &$\{0\}$    &$\{0\}$      &$\{0\}$
\end{tabular}
\end{center}
Of the corresponding polynomials 
$\varphi_{c,d}=x^9+\ngt{6}x^6+9c\cdot x^2+\ngt{18}x+3+\nrt{9d}$ $(c,d\in\{1,2\})$
more than one polynomial generates each extension.
Let $\alpha$ be  root of $\varphi_{c,d}$ and
$\rho$ its ramification polynomial .
For $\gamma\in\{\underline1,\underline2\}$ we have 
$v_\alpha(\rho(\gamma \alpha))=11$.
If $\psi(x)=\sum_{i=0}^9\psi_i x^i$ denotes the minimal polynomial of $\alpha+\gamma\alpha^2$ then 
by Proposition \ref{prop other uni} (a) we have $\varphi_2-\psi_2=\alpha^9\rho(\gamma\alpha)$.
and hence $\psi_{2,2}=\varphi_{2,2}-\rho(\gamma\alpha)/\alpha^9\not\equiv 0 \bmod \alpha$.
As $\gamma+(\alpha)\mapsto \rho(\gamma\alpha)/\alpha^{11}+(\alpha)=2\gamma+(\alpha)$ is surjective,
changing the uniformizer from $\alpha$ to $\alpha+\gamma\alpha$ results in a change of $\varphi_{2,2}$.
Thus we can choose $\gamma$ such that {$\varphi_{2,2}=0$} and get
that all extensions with ramification polygon $\rpol_3$ and residual polynomials $\invA_{3,2}$ are generated by exactly one polynomial 
of the form $\varphi_{d}=x^9+\ngt{6}x^6+\ngt{18}x+3+\nrt{9d}$ where $(d\in\{1,2\})$.
\end{example}

\begin{theorem}\label{theo alg}
Let $F$ be the set of polynomials returned by Algorithm \ref{alg all sub} 
given $\KK$ and a ramification polygon $\rpol$, $\underline\delta_0\in\RK$  and polynomials $\RA_1,\dots,\RA_\ell\in\RK[x]$.
\begin{enumerate}
\item 
$F$ contains at least one Eisenstein polynomial for each totally ramified extension of degree $n$,
that can be generated by a polynomial $\varphi$ with ramification polygon $\rpol$, 
$\underline\varphi_{0,1}=\underline\delta_0$, and residual polynomials $\RA_1,\dots,\RA_\ell$.
\item
If $\RS{m}:\RK\to\RK$ is surjective for all segments with integral slope $-m$,
then no two polynomials in $F$ generate isomorphic extensions.
\item
If there is exactly one $\RS{m}:\RK\to\RK$ that is non-surjective, and 
for all integers $k > n\hhf_\rpol(m)$, there is an $m'\in\N$ such that $n\hhf_\rpol(m') = k$,
then no two polynomials in $F$ generate isomorphic extensions.
\end{enumerate}
\end{theorem}

\begin{proof}
\begin{enumerate}
\item 
Let $\varphi\in F$.
In Algorithm \ref{alg all sub} step (c) we have ensured that $v_\pi(\varphi_i)\geq L_\rpol(i)$ and in step (e) we assign nonzero values to $\varphi_{b_i,j}$ so that $v_\pi(\varphi_{b_i}) = L_\rpol(b_i)$ for points $(p^{s_i},a_in+b_i)$ with $b_i\neq 0$. So by Proposition \ref{prop R iff}, $\varphi$ has ramification polygon $\rpol$.
By Lemma \ref{lem A fix}, the values assigned in step (e) ensure that $\rpol_\varphi$ has residual polynomials $(\RA_1,\dots,\RA_\ell)$.  
Thus each extension generated by a polynomial with the input invariants is generated by a polynomial in $F$ and all polynomials in $F$ have these invariants.
\item
If $\RS{m}:\RK\to\RK$ is surjective for all
segments with integral slope $-m$,
then all of the nonzero coefficients in our template $\tau$ are either fixed by $\delta_0$ or $\invA$, or
free because they are not set by a choice of element in the image of some $\RS{m}$.
Any deformation of the uniformizer that might result in two polynomials in $F$ to generate the same extension
would have to change one of these free coefficients, but such a change cannot be made independently of the
choices we made in order to set coefficients to zero by Lemma \ref{lem Sm choice}.
So no two polynomials in $F$ generate isomorphic extensions.

\item
Suppose there is exactly one $\RS{m}:\RK\to\RK$ that is non-surjective, and 
for all integers $k > n\hhf_\rpol(m)$, there is an $m'\in\N$ such that $n\hhf_\rpol(m') = k$.
As $\RS{m}:\RK\to\RK$ is non-surjective, there will be more than one choice for $\varphi_{i,j}$ where $jn+i = n\hhf_\rpol(m)$.
By Proposition \ref{prop other uni}, the corresponding change of uniformizer (from $\alpha$ to
$\alpha + \gamma\alpha^{m+1}$) can change $\varphi_{i',j'}$ where $j'n+i' > jn+i$.
Since there exists $m'\in\N$ such that $n\hhf_\rpol(m') = j'n+i'$,
then Algorithm \ref{alg all sub} will assign $\varphi_{i',j'}$ based on $\RS{m'}$.
Given that $m\neq m'$, $\RS{m'}$ is surjective, $\varphi_{i',j'}$ can be set to zero by Lemma \ref{lem Sm choice}.
As all coefficients $\varphi_{i',j'}$ with $j'n+i' \geq jn+i$ are assigned by the residual polynomials of components,
no two polynomials generate isomorphic extensions. \qedhere
\end{enumerate}
\end{proof}

As in general the algorithm returns more than one polynomial generating each extension with the given invariants, 
the output needs to be filtered by comparing the generated extensions by
\begin{enumerate}
\item computing all reduced generating polynomials using \cite[Algorithm 3]{monge} and comparing these or
\item using a root finding algorithm (compare \cite{pauli-roblot}).
\end{enumerate}
The product $\prod_{m=0}^\infty\#\ker\RS{m}$ is an upper bound for the number of 
automorphisms of $\KL/\KK$.  This together with the number of reduced polynomials of $\varphi$ gives
the number of automorphisms of $\KL/\KK$ (\cite[Theorem 1]{monge}).  Alternatively the number extensions 
generated by each polynomial can be computed using root finding.

Now we present an algorithm to enumerate all extensions with a given invariants.
It may require multiple calls to Algorithm \ref{alg all sub} \texttt{AllExtensionsSub} depending the structure of $\invA$
and the number of tame subextensions.

\Algo{AllExtensions} 
{alg all RA} 
{A $\pi$-adic field $\KK$, a ramification polygon $\rpol$, and invariant $\invA$
} 
{
A set $F$ that contains one generating Eisenstein polynomial for each totally ramified extension of $\KK$
with ramification polygon $\rpol$ and invariant $\invA$
} 
{
\begin{enumerate}
\item $S_0 \leftarrow$ a set of representatives of $\RK^\times/(\RK^\times)^n$.
\item For $\delta\in S_0$ do
    \begin{enumerate}
    \item Partition $\invA$ into disjoint sets $\invA^{*1},\ldots,\invA^{*k}$ by Equation (\ref{eq invAstar}).
    \item For $\invA^* \in \{\invA^{*1},\ldots,\invA^{*k}\}$ do
    \begin{itemize}
        \item Let $A$ be a representative of $\invA^*$.
        \item $F' \leftarrow$ \texttt{AllExtensionsSub}($\KK,\rpol,A,\delta$). \hfill[Alg. \ref{alg all sub}]
        \item Unless avoidable by Theorem \ref{theo alg},
            filter $F'$ so that no two polynomials generate the same extension using method of choice.
        \item $F\leftarrow F\cup F'$.
    \end{itemize}
    \end{enumerate}
\item Return $F$.
\end{enumerate}
}

\begin{theorem}
Let $F$ be the set of polynomials returned by Algorithm \ref{alg all RA}.
For each extension $\KL/\KK$ with ramification polygon $\rpol$ and invariant $\invA$,
the set $F$ contains exactly one generating polynomial.
\end{theorem}
\begin{proof}
Let $\KL/\KK$ be a totally ramified extension with ramification polygon $\rpol$ and invariant $\invA$.
Let $\psi\in\OK[x]$ be an Eisenstein polynomial generating $\KL$ with $\psi_{0,1}\in S_0$.
Let $A^{(\psi)}$ be the residual polynomials of segments of $\rpol$ given $\psi$.
As $\psi$ generates $\KL$ with invariant $\invA$,
$A^{(\psi)}$ belongs to some $\invA^*$ in our partition of $\invA$.
If $A$ is our choice of representative of $\invA^*$, then by Lemma \ref{lem A star}, there is a $\varphi\in\OK[x]$ with residual polynomials $A$ such that $\KK[x]/(\psi)\cong\KK[x]/(\varphi)$.
Thus, $\KL/\KK$ can be generated by an Eisenstein polynomial $\varphi$ with residual polynomials $A$, and $\varphi_{0,1}=\psi_{0,1}$, and by Theorem \ref{theo alg}, there is at least one $\varphi\in F'$ with $F'$ returned by
\texttt{AllExtensionsSub}($\KK,\rpol_\psi,A,\psi_{0,1}$) generating $\KL/\KK$.
The output $F$ contains one generator for every extension that can be generated by any
polynomial in any $F'$ produced, and so there is a polynomial in $F$ generating $\KL/\KK$.

To show that no two polynomials in $F$ generate the same extension, it suffices to show that
no polynomials produced by different calls to Algorithm \ref{alg all sub} generate the same extension.
Let $\varphi$ and $\psi$ be in two such polynomials.
By Lemma \ref{lem phi0}, if $\varphi_{0,1}\neq\psi_{0,1}$, then as
$\varphi_{0,1},\psi_{0,1}\in\RK^\times/(\RK^\times)^n$, $\KK[x]/(\psi)\ncong\KK[x]/(\varphi)$.
Now suppose $\varphi_{0,1} = \psi_{0,1}$.
By Remark \ref{rem A star}, if the residual polynomials of $\varphi$ and $\psi$ are not in
the same $\invA^*$ then $\KK[x]/(\psi)\ncong\KK[x]/(\varphi)$.
Thus, if two polynomials are generated by Algorithm \ref{alg all sub} with different inputs of
$\delta$ or residual polynomials returned by Algorithm \ref{alg all RA},
they cannot generate the same extension.
\end{proof}


\section{Examples}\label{sec ex}

In Figure \ref{fig compare} we compare the implementation of
the algorithm from \cite{pauli-roblot} in Magma \cite{magma} (\texttt{AllExtensions}) and Pari \cite{pari} (\texttt{padicfields})
with our implementation of Algorithm {\ref{alg all RA}}  in Magma using root finding to filter the set of polynomials to obtain a minimal set.
In the implementation of the method from \cite{pauli-roblot}  Magma we 
replaced the deterministic enumeration of polynomials by random choices, 
which yields a considerable performance improvement.
In our implementation of Algorithm {\ref{alg all RA}} the filtering out of redundant polynomials 
can accelerated by using reduction \cite{monge} instead of root finding.

\begin{figure}
\begin{center}
\renewcommand\arraystretch{1.1}
\renewcommand\tabcolsep{2pt}
\begin{tabular}{l|c|c|c||c|c||c}
$\KK$ & $n$ & $v(\disc)$ & $\#F$ & Magma \cite{pauli-roblot} & Pari \cite{pauli-roblot} & Magma (Alg. \ref{alg all RA})\\ \hline
$\Q_3$& 9  & 9   &  2 & 10 ms     & 37 ms     & 10 ms\\
$\Q_3$& 9  & 22  & 96 & 67 s & 11 s & 30 ms + 5.77 s$\dagger$\\
$\Q_3$& 9  & 26  & 81 & 16.61 s   & 3.64 s    & 0.05 s\\
$\Q_3$& 27 & 27  &  2 & 30 ms     & 56 h      & 10 ms\\
$\Q_3$& 27 & 107 & 1,594,323 & $>5$ days      & --- & 17 min \\
\end{tabular}
\end{center}
\caption{Time needed to compute a minimal set $F$ of generating polynomials of all extensions of $\KK$ of 
degree $n$ with discriminant exponent $v(\disc)$.
All timings were obtained on a computer with a Intel Core 2 Quad CPU at 2.83GHz and 8Gb RAM running Ubuntu Linux 14.04 LTS.
($\dagger$ time required to filter output of Alg. \ref{alg all sub})}
\label{fig compare}
\end{figure}

We now present generating polynomials for 
totally ramified extensions of degree 15 over $\Q_5$ (Example \ref{ex q5 15}),
totally ramified extensions of degree 8 over an unramified extension of degree 2 over $\Q_2$ (Example \ref{ex q2 unram2 8}),
totally ramified extensions of degree 9 over a ramified extension of $\Q_3$ of degree 3 (Example \ref{ex q3 ram3 9}),
and an example over $\Q_3$ that shows that in general not all extensions with the same ramification polygon and invariant $\invA$ have the same mass
(Example \ref{ex q3 9 mass}).

\begin{example}\label{ex q5 15}
We find generating polynomials for all totally ramified 
extensions $\KL$ of $\Q_5$ of degree 15 with $v_5(\disc(\KL))=29$,
the highest possible valuation by Proposition \ref{prop.ore}.
There is only one possible ramification polygon 
$\rpol = \{(1,15),(5,0),(10,0),(15,0)\}$ and only one possible set
of residual polynomials $\invA = \{ (3z + 2, z^{10} + 3z^5 + 3) \}$
for such extensions.
Denote by $\varphi(x)=\sum_{i=0}^{15}\varphi_i x^i$ an Eisenstein polynomial generating such a field $\KL$.

By Lemma \ref{lem phi0} all extensions of $\Q_5$ with ramification polygon $\rpol$ can
be generated by polynomials $\varphi\in\Z_5[x]$ with $\varphi_0\equiv 5\bmod 25$.
As $b_t = 0$ for all points $(p^{s_t},a_tn+b_t)\in\rpol$, Proposition \ref{prop R iff}
only gives us restrictions on $\varphi$ based on $L_\rpol$ and no coefficients
are set by Lemma \ref{lem A fix}.
This provides the following template for $\varphi$:

\begin{center}
\footnotesize
\renewcommand\arraystretch{1.1}
\renewcommand\tabcolsep{2pt}
\begin{tabular}{l|ccccccccccccccccc}
&$x^{15}$ &$x^{14}$ &$x^{13}$ &$x^{12}$ &$x^{11}$ &$x^{10}$ &$x^{9}$ &$x^{8}$ &$x^{7}$ &$x^{6}$ &$x^{5}$ &$x^{4}$ &$x^{3}$ &$x^{2}$ &$x^{1}$ &$x^{0}$ \\\hline
$5^2$&$\{0\}$ &$\reps{\F_5}$ &$\reps{\F_5}$ &$\reps{\F_5}$ &$\reps{\F_5}$ &$\reps{\F_5}$ &$\reps{\F_5}$ &$\reps{\F_5}$ &$\reps{\F_5}$ &$\reps{\F_5}$ &$\reps{\F_5}$ &$\reps{\F_5}$ &$\reps{\F_5}$ &$\reps{\F_5}$ &$\reps{\F_5}$ &$\reps{\F_5}$ \\
$5^1$&$\{0\}$ &$\{0\}$ &$\{0\}$ &$\{0\}$ &$\{0\}$ &$\reps{\F_5}$ &$\{0\}$ &$\{0\}$ &$\{0\}$ &$\{0\}$ &$\reps{\F_5}$ &$\{0\}$ &$\{0\}$ &$\{0\}$ &$\{0\}$ &$\{1\}$ \\
$5^0$&$\{1\}$ &$\{0\}$ &$\{0\}$ &$\{0\}$ &$\{0\}$ &$\{0\}$ &$\{0\}$ &$\{0\}$ &$\{0\}$ &$\{0\}$ &$\{0\}$ &$\{0\}$ &$\{0\}$ &$\{0\}$ &$\{0\}$ &$\{0\}$
\end{tabular}
\end{center}

The ramification polygon $\rpol_2$ has no segments with non-zero integral slope.
We get $\RS{1}=x^{15}$, $\RS{2}=x^{15}$, and $\RS{3}=x^{15}$, 
with $15\hhf(1)=5$, $15\hhf(2)=10$, and $15\hhf(3)=15$.
Thus $\varphi_{5,1}=0$, {$\varphi_{10,1}=0$}, and {$\varphi_{0,2}=0$}.
Further, for $m\ge 4$, $\RS{m}=x$. As $15\hhf(m)=15+m$ for $m\ge 4$,
by Lemma \ref{lem Sm choice}, we can set {$\varphi_{k,j}=0$} 
for $k+9(j-1)\ge 19$.
Therefore, the generating polynomials $\varphi$ of the fields over $\Q_5$
with invariants $\rpol$ and $\invA$ follow this template:

\begin{center}
\footnotesize
\renewcommand\arraystretch{1.1}
\renewcommand\tabcolsep{2pt}
\begin{tabular}{l|ccccccccccccccccc}
&$x^{15}$ &$x^{14}$ &$x^{13}$ &$x^{12}$ &$x^{11}$ &$x^{10}$ &$x^{9}$ &$x^{8}$ &$x^{7}$ &$x^{6}$ &$x^{5}$ &$x^{4}$ &$x^{3}$ &$x^{2}$ &$x^{1}$ &$x^{0}$ \\\hline
$5^2$&$\{0\}$ &$\{0\}$ &$\{0\}$ &$\{0\}$ &$\{0\}$ &$\{0\}$ &$\{0\}$ &$\{0\}$ &$\{0\}$ &$\{0\}$ &$\{0\}$ &$\{0\}$ &$\reps{\F_5}$ &$\reps{\F_5}$ &$\reps{\F_5}$ &$\{0\}$ \\
$5^1$&$\{0\}$ &$\{0\}$ &$\{0\}$ &$\{0\}$ &$\{0\}$ &$\{0\}$ &$\{0\}$ &$\{0\}$ &$\{0\}$ &$\{0\}$ &$\{0\}$ &$\{0\}$ &$\{0\}$ &$\{0\}$ &$\{0\}$ &$\{1\}$ \\
$5^0$&$\{1\}$ &$\{0\}$ &$\{0\}$ &$\{0\}$ &$\{0\}$ &$\{0\}$ &$\{0\}$ &$\{0\}$ &$\{0\}$ &$\{0\}$ &$\{0\}$ &$\{0\}$ &$\{0\}$ &$\{0\}$ &$\{0\}$ &$\{0\}$
\end{tabular}
\end{center}
As all of the $\RS{m}$ are surjective, by Theorem \ref{theo alg} (b), no two of these 125 polynomials generate isomorphic extensions of $\Q_5$.
\end{example}

\begin{example}\label{ex q2 unram2 8}
Let $\KK$ be the unramified extension of $\Q_2$ generated by $y^2+y+1\in\Q_2[y]$.
Let $\gamma$ be a root of $y^2+y+1$, so $\RK = \F_2(\gamma)$.
We want to find generating polynomials for all totally ramified 
extensions $\KL$ of $\KK$ of degree 8 with $v_2(\disc(\KL))=16$,
ramification polygon with points $\rpol = \{(1,9),(2,6),(8,0)\}$,
and $\invA$ containing $(\gamma z + \gamma, z^6 + \gamma)$.
Denote by $\varphi=\sum_{i=0}^8\varphi_i x^i$ an Eisenstein polynomial generating such a field $\KL$.

By Proposition \ref{prop R iff}, we have $v(\varphi_1)=2$ and $v(\varphi_6)=1$,
and that $v(\varphi_i)\geq 2$ for $i\in\{2,3,4,5,7\}$.
By Lemma \ref{lem A fix},
the point $(1,9)=(2^0,1\cdot 8+1)$ on $\rpol$ gives us that $\varphi_{1,2}=\gamma$ and
the point $(2,6)=(2^1,0\cdot 8+6)$ on $\rpol$ gives us that $\varphi_{6,1}=\gamma$.
We set $\varphi_{0,1}=1$ by Lemma \ref{lem phi0}
and the template for the polynomials $\varphi$ is:

\begin{center}
\small
\renewcommand\arraystretch{1.1}
\renewcommand\tabcolsep{2pt}
\begin{tabular}{l|cccccccccc}
     &$x^{8}$ &$x^{7}$ &$x^{6}$   &$x^{5}$ &$x^{4}$   &$x^{3}$ &$x^{2}$   &$x^{1}$   &$x^{0}$   \\\hline
$2^3$&$\{0\}$ &$\reps{\RK}$   &$\reps{\RK}$     &$\reps{\RK}$   &$\reps{\RK}$     &$\reps{\RK}$   &$\reps{\RK}$     &$\reps{\RK}$     &$\reps{\RK}$     \\
$2^2$&$\{0\}$ &$\reps{\RK}$   &$\reps{\RK}$     &$\reps{\RK}$   &$\reps{\RK}$     &$\reps{\RK}$   &$\reps{\RK}$     &$\{\gamma\}$ &$\reps{\RK}$ \\
$2^1$&$\{0\}$ &$\{0\}$ &$\{\gamma\}$ &$\{0\}$ &$\{0\}$   &$\{0\}$ &$\{0\}$   &$\{ 0 \}$ &$\{1\}$ \\
$2^0$&$\{1\}$ &$\{0\}$ &$\{0\}$ &$\{0\}$ &$\{0\}$ &$\{0\}$ &$\{0\}$ &$\{0\}$ &$\{0\}$ 
\end{tabular}
\end{center}

It remains to consider the $\RS{m}$.
Our ramification polygon $\rpol$ has two segments of integral slope, $-3$ and $-1$, respectively.
So by Lemma \ref{lem comp res poly}, $\RS{1}(z)=z^2\RA_2 = z^2(z^6 + \gamma)$ and $\RS{3}(z)=z\RA_1 = z(\gamma z + \gamma)$.
As $\RS{1}$ is surjective and $n\hhf(1)=8$, we may set $\varphi_{0,2} = 0$.
As $\rpol$ has no segment of slope $-2$, $\RS{2}$ is surjective, so with $n\hhf(2)=10$, we may set $\varphi_{2,2} = 0$.
On the other hand, $\RS{3}$ is not surjective and has image $\{0,\gamma\}$.
By Lemma \ref{lem Sm choice} and as $n\hhf(3)=12$, $\varphi_{4,2}\in\reps{\RK}/\{0,\gamma\}=\{0,1\}$.
For $m\geq 4$, $n\hhf(m)=9+m$, and so we can set {$\varphi_{k,j}=0$} for $k+8(j-1)\ge 13$.
This gives us the following template for polynomials $\varphi$:

\begin{center}
\small
\renewcommand\arraystretch{1.1}
\renewcommand\tabcolsep{2pt}
\begin{tabular}{l|cccccccccc}
             &$x^8$    &$x^7$      &$x^6$        &$x^5$   &$x^4$      &$x^3$      &$x^2$      &$x^1$        &$x^0$ \\\hline
      $2^3$&$\{0\}$    &$\{0\}$    &$\{0\}$      &$\{0\}$ &$\{0\}$    &$\{0\}$    &$\{0\}$    &$\{0\}$      &$\{0\}$  \\
      $2^2$&$\{0\}$    &$\{0\}$    &$\{0\}$      &$\{0\}$ &$\{0,1\}$  &$\repsK$   &$\{0\}$    &$\{\gamma\}$ &$\{0\}$  \\
      $2^1$&$\{0\}$    &$\{0\}$    &$\{\gamma\}$ &$\{0\}$ &$\{0\}$    &$\{0\}$    &$\{0\}$    &$\{0\}$      &$\{1\}$  \\
      $2^0$&$\{1\}$    &$\{0\}$    &$\{0\}$      &$\{0\}$ &$\{0\}$    &$\{0\}$    &$\{0\}$    &$\{0\}$      &$\{0\}$
\end{tabular}
\end{center}

As $\RS{3}$ is the only non-surjective $\RS{m}$, and for all integers $k$ greater than $n\hhf(3)=12$, $n\hhf(k-9)=k$,
we have by Theorem \ref{theo alg} (c) that no two of these 8 polynomials generate the same extension.
\end{example}

\begin{example}\label{ex q3 ram3 9}
Let $\KK = \Q_3[x]/(x^2-3)$ and let $\pi$ be a uniformizer of the valuation ring of $\KK$.
As in Example \ref{ex R2}, there are three possible ramification polygons for extensions $\KL$ of 
$\KK$ of degree $9$ with $v_3(\disc(\KL))=18$, namely
$\rpol_1=\{(1,10),(9,0)\}$,
$\rpol_2=\{(1,10),(3,3),(9,0)\}$, and
$\rpol_3=\{(1,10),(3,6),(9,0)\}$ (compare Figure \ref{fig ram pol ex}).

Let us again choose to investigate $\rpol_2$.
By Lemma \ref{lem vphi} we have {$v_\pi(\varphi_3)=1$}
and by Lemma \ref{lem phi0} we can set $\varphi_{0,1}=1$.
As $\RK=\underline{\smash{\Q_3}}$, we have the same four choices 
for the invariant $\invA$:
$\invA_{2,1}=\{(1+2x, 2+x^3)\}$,
$\invA_{2,2}=\{(2+x, 1+2x^3)\}$,
$\invA_{2,3}=\{(1+x, 1+x^3)\}$, and
$\invA_{2,4}=\{(2+2x, 2+x^3)\}$.

Let us choose $\invA_{2,1}$.
By Lemma \ref{lem A fix} we get
from the point $(1,10)=(3^0,1\cdot9+1)$ on $\rpol_2$ that {$\varphi_{1,2}=1$} and 
from the point $(3,3)=(3^1,0\cdot9+3)$ on $\rpol_2$ that {$\varphi_{3,1}=2$}.

The ramification polygon $\rpol_2$ has no segments with integral slope.
We get $\RS{1}=x^3$, $\RS{2}=x^3$, and $\RS{3}=x^3$, 
with $9\hhf(1)=6$, $9\hhf(2)=9$, and $9\hhf(3)=12$.
Thus {$\varphi_{6,1}=0$}, {$\varphi_{0,2}=0$}, and {$\varphi_{3,2}=0$}.
Furthermore $\RS{m}=x$ for with $9\hhf(m)=10+m$ for $m\ge 4$.
Thus by Lemma \ref{lem Sm choice} we can set {$\varphi_{k,j}=0$} 
for $k+9(j-1)\ge 14$.

Proceeding as in  Examples \ref{ex R2}, \ref{ex R2 A21}, and \ref{ex R2 A21 S} we obtain
a familiar template for the polynomials generating fields over $\KK$ with ramification polygon $\rpol_2$
and invariant $\invA_{2,1}$:

\begin{center}
\small
\renewcommand\arraystretch{1.1}
\renewcommand\tabcolsep{2pt}
\begin{tabular}{l|cccccccccc}
             &$x^9$   &$x^8$   &$x^7$   &$x^6$  &$x^5$  &$x^4$      &$x^3$   &$x^2$       &$x^1$   &$x^0$    \\\hline
      $\pi^4$&$\{0\}$ &$\{0\}$ &$\{0\}$ &$\{0\}$&$\{0\}$&$\{0\}$    &$\{0\}$ &$\{0\}$     &$\{0\}$ &$\{0\}$  \\
      $\pi^3$&$\{0\}$ &$\{0\}$ &$\{0\}$ &$\{0\}$&$\{0\}$&$\{0\}$    &$\{0\}$ &$\{0\}$     &$\{0\}$ &$\{0\}$  \\
      $\pi^2$&$\{0\}$ &$\{0\}$ &$\{0\}$ &$\{0\}$&$\{0\}$&$\{0,1,2\}$&$\{0\}$ &$\{0,1,2\}$ &$\{1\}$ &$\{0\}$  \\
      $\pi^1$&$\{0\}$ &$\{0\}$ &$\{0\}$ &$\{0\}$&$\{0\}$&$\{0\}$    &$\{2\}$ &$\{0\}$     &$\{0\}$ &$\{1\}$  \\
      $\pi^0$&$\{1\}$ &$\{0\}$ &$\{0\}$ &$\{0\}$&$\{0\}$&$\{0\}$    &$\{0\}$ &$\{0\}$     &$\{0\}$ &$\{0\}$
\end{tabular}
\end{center}

As all of the $\RS{m}$ are surjective, we obtain a unique generating polynomial
of each degree 9 extension of $\KK$ with $v_3(\disc(\KL))=18$, ramification polygon $\rpol_2$, and invariant $\invA_{2,1}$.

\end{example}

As mentioned in the previous section, our choice of residual polynomials relate to the size of the automorphism group
of the extensions generated by our polynomials.
However, the polynomials generated by Algorithm {\ref{alg all RA}} (and in general, those generating extensions of the same degree,
discriminant, ramification polygon, and $\invA$) do not generate extensions with the same automorphism group size.

\begin{example}\label{ex q3 9 mass}
Over $\Q_3[x]$, let $\varphi(x) = x^9 + 6x^6 + 18x^5 + 3$ and $\psi(x) = x^9 + 18x^8 + 9x^7 + 6x^6 + 18x^5 + 3$.
Both are Eisenstein polynomials generating degree $9$ extensions over $\Q_3$ with
ramification polygon $\rpol = \{(1,14),(3,6),(9,0)\}$ and having residual polynomials $\RA_1 = 2z^2+1$ and $\RA_2 = z^6+2$.
Using root-finding, we see that over $\Q_3[x]/(\varphi)$, $\varphi$ has 3 roots, while over $\Q_3[x]/(\psi)$, $\psi$ has 9 roots.
Thus $\psi$ generates a normal extension, while $\varphi$ generates three extensions with automorphism groups of size 3 which shows that not all extension with the same ramification polygon and residual polynomials have the same mass.
\end{example}


\section{Acknowledgments}\label{sec ack}

We thank Jonathan Milstead for his careful reading of our manuscript.

\bibliography{local}

\providecommand{\bysame}{\leavevmode\hbox to3em{\hrulefill}\thinspace}
\providecommand{\MR}{\relax\ifhmode\unskip\space\fi MR }
\providecommand{\MRhref}[2]{%
  \href{http://www.ams.org/mathscinet-getitem?mr=#1}{#2}
}
\providecommand{\href}[2]{#2}
\begin{thebibliography}{GMN13}

\bibitem[Ama71]{amano}
Shigeru Amano, \emph{Eisenstein equations of degree {$p$} in a
  {${\mathfrak{p}}$}-adic field}, J. Fac. Sci. Univ. Tokyo Sect. IA Math.
  \textbf{18} (1971), 1--21. \MR{0308086 (46 \#7201)}

\bibitem[BCP97]{magma}
Wieb Bosma, John Cannon, and Catherine Playoust, \emph{The {M}agma algebra
  system. {I}. {T}he user language}, J. Symbolic Comput. \textbf{24} (1997),
  no.~3-4, 235--265, Computational algebra and number theory (London, 1993).
  \MR{MR1484478}

\bibitem[FV02]{fv}
Ivan~B. Fesenko and Sergey~V. Vostokov, \emph{Local fields and their
  extensions}, 2nd ed., Translations of Mathematical Monographs, vol. 121,
  American Mathematical Society, 2002.

\bibitem[GMN13]{guardia-montes-nart}
Jordi Gu{\`a}rdia, Jes{\'u}s Montes, and Enric Nart, \emph{A new computational
  approach to ideal theory in number fields}, Found. Comput. Math. \textbf{13}
  (2013), no.~5, 729--762. \MR{3105943}

\bibitem[GNP12]{guardia-nart-pauli}
Jordi Gu{\`a}rdia, Enric Nart, and Sebastian Pauli, \emph{Single-factor lifting
  and factorization of polynomials over local fields}, J. Symbolic Comput.
  \textbf{47} (2012), no.~11, 1318--1346. \MR{2927133}

\bibitem[GP12]{greve-pauli}
Christian Greve and Sebastian Pauli, \emph{Ramification polygons, splitting
  fields, and {G}alois groups of {E}isenstein polynomials}, International
  Journal of Number Theory \textbf{8} (2012), no.~6, 1401--1424. \MR{2965757}

\bibitem[Hel90]{helou}
Charles Helou, \emph{Non-{G}alois ramification theory of local fields}, Algebra
  Berichte [Algebra Reports], vol.~64, Verlag Reinhard Fischer, Munich, 1990.
  \MR{1076620 (91j:11103)}

\bibitem[Kra66]{krasner}
Marc Krasner, \emph{Nombre des extensions d'un degr\'e donn\'e d'un corps
  {${\mathfrak{p}}$}-adique}, Les {T}endances {G}\'eom. en {A}lg\`ebre et
  {T}h\'eorie des {N}ombres, Editions du Centre National de la Recherche
  Scientifique, Paris, 1966, pp.~143--169. \MR{0225756 (37 \#1349)}

\bibitem[Li97]{li}
Hua-Chieh Li, \emph{p-adic power series which commute under composition},
  Transactions of the American Mathematical Society \textbf{349} (1997), no.~4,
  1437--1446.

\bibitem[Lub81]{lubin}
Jonathan~D. Lubin, \emph{The local {K}ronecker-{W}eber theorem}, Transactions
  of the American Mathematical Society \textbf{267} (1981), no.~1, 133--138.

\bibitem[MN92]{montes-nart}
Jes{\'u}s Montes and Enric Nart, \emph{On a theorem of {O}re}, J. Algebra
  \textbf{146} (1992), no.~2, 318--334. \MR{1152908 (93f:11077)}

\bibitem[Mon99]{montes}
Jes{\'u}s Montes, \emph{Poligonos de {N}ewton de orden superior y aplicaciones
  aritmeticas}, 1999, Thesis (Ph.D.)--Universitat de Barcelona.

\bibitem[Mon14]{monge}
Maurizio Monge, \emph{A family of {E}isenstein polynomials generating totally
  ramified extensions, identification of extensions and construction of class
  fields}, Int. J. Number Theory \textbf{10} (2014), no.~7, 1699--1727.
  \MR{3256847}

\bibitem[Ore26]{ore-bemerkungen}
{\"O}ystein Ore, \emph{Bemerkungen zur {T}heorie der {D}ifferente}, Math. Z.
  \textbf{25} (1926), no.~1, 1--8. \MR{1544795}

\bibitem[Ore28]{ore-newton}
\bysame, \emph{Newtonsche {P}olygone in der {T}heorie der algebraischen
  {K}\"orper}, Math. Ann. \textbf{99} (1928), no.~1, 84--117 (German).
  \MR{1512440}

\bibitem[Pan95]{panayi}
Peter Panayi, \emph{Computation of {L}eopoldt's $p$-adic regulator}, 1995,
  Thesis (Ph.D.)--University of East Anglia.

\bibitem[Pau06]{pauli-lcf}
Sebastian Pauli, \emph{Constructing class fields over local fields}, J.
  Th\'eor. Nombres Bordeaux \textbf{18} (2006), no.~3, 627--652. \MR{2330432
  (2008f:11135)}

\bibitem[Pau10]{pauli-pf2}
\bysame, \emph{Factoring polynomials over local fields {II}}, Algorithmic
  number theory, Lecture Notes in Comput. Sci., vol. 6197, Springer, Berlin,
  2010, pp.~301--315. \MR{2721428 (2012c:12002)}

\bibitem[PG14]{pari}
The Pari~Group, \emph{{Pari/GP version {\tt 2.7.0}}}, Bordeaux, 2014, available
  from \texttt{http://pari.math.u-bordeaux.fr/}.

\bibitem[PR01]{pauli-roblot}
Sebastian Pauli and Xavier-Fran{\c{c}}ois Roblot, \emph{On the computation of
  all extensions of a {$p$}-adic field of a given degree}, Math. Comp.
  \textbf{70} (2001), no.~236, 1641--1659 (electronic). \MR{1836924
  (2002e:11166)}

\bibitem[Sch03]{scherk}
John Scherk, \emph{The ramification polygon for curves over a finite field},
  Canad. Math. Bull. \textbf{46} (2003), no.~1, 149--156. \MR{1955622
  (2004c:11220)}

\bibitem[Ser79]{serre}
Jean-Pierre Serre, \emph{Local fields}, Graduate Texts in Mathematics, vol.~67,
  Springer-Verlag, New York-Berlin, 1979, Translated from the French by Marvin
  Jay Greenberg. \MR{554237 (82e:12016)}

\end{thebibliography}
\bibliographystyle{amsalpha}

\end{document}